\documentclass[11pt]{amsart}
 \usepackage[backref=page]{hyperref}
\hypersetup{
  colorlinks=true, 
   linkcolor=red,  
}

\usepackage{xypic}

\usepackage[enableskew]{youngtab}

\usepackage{caption}
\usepackage{subcaption}

\theoremstyle{plain}
 
\newtheorem{thm}{Theorem} 
\newtheorem{theorem}[thm]{Theorem}
 
\newtheorem{remark}[thm]{Remark} 
\newtheorem{lemma}[thm]{Lemma}
\newtheorem{corollary}[thm]{Corollary} 
 
\newtheorem{question}[thm]{Open Problem} 
 
\newtheorem{conjecture}[thm]{Conjecture}

\numberwithin{thm}{section}
 
\def\e{\mathbb{E}\, }
\def\ex{\mathbb{E}_x\, }
\def\ext{\mathbb{E}_{\theta,x}\, }
\def\p{\mathbb{P} }
\def\BN{\mathbb{N} }
\def\BR{\mathbb{R} }
\def\BZ{\mathbb{Z} }
\def\BF{\mathbb{F} }

\def\dtv{d_{\rm TV}}

\def\cL{\mathcal{L}}

\def\N{{\overline{N}} } 
\newcommand{\ignore}[1]{ }

\def\bC{{\mathbf{C}}}
\def\bD{{\mathbf{D}}}
\def\ba{{\mathbf{a}}}

\newcommand{\baj}[1]{{\ba^{(#1)}}}

\def\.{\hskip.06cm}

\title[PIPARCS]{
Poisson and independent process approximation for random combinatorial structures  with a given number of components, and near-universal behavior for low rank assemblies
}  

\author{Richard Arratia and Stephen DeSalvo}
\date{July 4, 2016}  

\def \ba {{\bf a}}

\begin{document}
\begin{abstract}
We give a general framework for approximations to combinatorial assemblies, especially suitable to the situation where the number $k$ of components is specified, in addition to the overall size $n$.   
This involves a Poisson process,   which, with the appropriate choice of parameter,  may be viewed as an extension of saddlepoint approximation.

We illustrate the use of this by analyzing the component structure when the rank and size are specified, and the rank,  $r := n-k$, is small relative to $n$.   There is near-universal behavior, in the sense that apart from cases where the exponential generating function has radius of convergence zero,   for $\ell=1,2,\dots$, when   $r \asymp n^\alpha$ for fixed $\alpha \in (\frac{\ell}{\ell+1}, \frac{\ell+1}{\ell+2})$, the size $L_1$ of the largest component converges in probabiity to $\ell+2$.  Further, when $r \sim t n^{\ell/(\ell+1)}$ for a positive integer $\ell$,  and $t \in (0,\infty)$,   $\p(L_1 \in \{\ell+1,\ell+2\}) \to 1$, with the choice governed by   a Poisson limit distribution for the number of components of size $\ell+2$.   This was previously observed, for the case $\ell=1$ and the special cases of permutations and set partitions, using Chen-Stein approximations for the indicators of attacks and alignments, when rooks are placed randomly on a triangular board.  The case $\ell=1$ is especially delicate, and was not handled by previous saddlepoint approximations.

\end{abstract}

\maketitle

\tableofcontents

\section{Introduction}\label{sect intro}
 
\subsection{Overview of this paper} 
  There are good independent process approximations for random decomposable combinatorial structures, including the broad classes of assemblies, multisets, and selections; see~\cite{IPARCS}.
In particular for assemblies, though not for the the other two, the approximating independent variables optimally come from a Poisson process. 
This specific structure leads to an especially effective representation of assemblies with a given number of components, as summarized in \eqref{Y} -- \eqref{CRY T}.  

The plan of this paper is as follows:  We will first review conditional
independent process descriptions of general combinatorial objects. Then,
specializing to assemblies, we will describe the additional Poisson structure,
and the ensuing saddlepoint analysis.  We conclude with a specific application,
a universal limit law for low rank assemblies. 
The conclusions of this low rank limit law do \emph{not} apply to multisets or
selections: while polynomials over finite fields are very close to random permutations
in many properties of the joint distribution of the sizes of the parts, the \emph{conditional}
distributions, conditioned on the existence of very many parts, are quite different.

Here is an overview of the universal low rank limit for assemblies.   First,  Theorem  \ref{thm 1} handles the key critical case,  picking an assembly uniformly at random from the $p(n,k)$ instances of size $n$ having exactly $k$ components, in a regime where,  for some fixed $t>0$,  $n$ and $k$ tend to infinity together with $n-k \sim t \sqrt{n}.$ 
 (\cite[Theorem 3]{granville}  gives some information about a very different large deviation regime,  including cases where $k \sim t n^\alpha$ for fixed $\alpha \in (0,1/2)$.)
The behavior is that the largest component has, with high probability, size 2 or 3;   the number of components of size 3 has a limit Poisson distribution.   This case cannot be handled by traditional saddlepoint analysis,  where Gaussian behavior plays the crucial role.

\ignore{ 
Two special assemblies are set partitions, and permutations,  so that $p(n,k)$ is a Stirling number of the second or first kind.    For these two cases, the
Poissonian behavior when $n-k \sim t \sqrt{n}$ was discovered in \cite{Stirling}  and \cite{Stirling2}, and the limit behavior was derived  from Chen-Stein analysis of combinatorial bijections,  involving $n-k$ rooks placed on a triangular board of size $n$.
For most aspects of component behavior,   permutations and set partitions have very different behavior,  for example a random permutation tends to have around $\log n$ cycles  but a random set partition tends to have around $n/\log n$ blocks.   The \emph{only} thing permutations and set partitions have in common is the structure of an \emph{assembly};  the Poissonian  low rank limit law does \emph{not} apply to multisets, such as  random polynomials over a finite field.   Strikingly,   random polynomials over a finite field are very close to random permutations, with respect to most aspects of the behavior of the distribution of part sizes,  but not with respect to that aspect  where the number of parts $k$ is conditioned to be so  large that  $n-k$  grows like $n^\alpha$ for fixed $\alpha \in (0,1$).    
}

Two special assemblies are set partitions, and permutations,  so that $p(n,k)$ is a Stirling number of the second or first kind.    For these two cases, the
Poissonian behavior when $n-k \sim t \sqrt{n}$ was discovered in \cite{Stirling}  and \cite{Stirling2}, and the limit behavior was derived  from Chen-Stein analysis of combinatorial bijections,  involving $n-k$ rooks placed on a triangular board of size $n$.
For most aspects of component behavior,   permutations and set partitions have very different behavior,  for example a random permutation tends to have around $\log n$ cycles  but a random set partition tends to have around $n/\log n$ blocks.   The \emph{only} thing permutations and set partitions have in common is the structure of an \emph{assembly};  the Poissonian  low rank limit law does \emph{not} apply to multisets, such as  random polynomials over a finite field.

Theorem \ref{thm 1}  has an asymptotic error bound of the form $O_t(\log^2 n/\sqrt{n})$.   It is essential to be aware that for a specific instance of $n$ and $k$, big O asymptotics provide no information at all.   Theorem \ref{SD thm 1}
carries out essentially the same analysis, but gives a quantitative, i.e., completely effective, 
error bound, as in   \cite{rosser}.  
Thus, the bound in  Theorem \ref{SD thm 1} is a complicated  but explicit function $u_M(n,k)$ of $n$ and $k$, and the parameters of the assembly, with the property that under the regime with $n-k \sim t \sqrt{n}$, asymptotically we have $u_M(n,k) = O_t(\log^2 n/\sqrt{n})$ --- thereby giving an alternate proof of  Theorem \ref{thm 1}.

We round out the analysis of low rank assemblies with theorems treating the low rank behavior on both sides of the critical regime  $n-k \sim t \sqrt{n}$.   
On the side where $n-k$ is smaller, handling the case $n-k \to \infty$ with $(n-k)/\sqrt{n} \to 0$, the error bound in Theorem~\ref{thm 1.5} is quantitative \emph{and} asymptotically sharp.  On the side where  $n-k$ is larger,  Theorem \ref{thm 2} handles the case where  $n-k \asymp n^\alpha$ for fixed $\alpha \in (1/2,1)$.   For each $\ell=2,3,\dots$,  there is a critical case:  when $\alpha = \ell/(\ell+1)$,  that the largest component has, with high probability, size $\ell+1$ or $\ell+2$;   the number of components of size $\ell+2$ has a limit Poisson distribution.  

\subsection{Connections with enumerative combinatorics}

Our results are parallel to classical enumerative combinatorics, where there has been much interest in counting the number of such restricted structures and proving certain smoothness conditions like unimodality and log-concavity. 
Much of the theoretical treatment centers on saddle point analysis, whereas our approach is distinctly probabilistic. 
Nevertheless, one sees the same quantities appear both in the classical analysis of generating functions as well as the probabilistic treatment involving conditional distributions. 
The goal is the same, which is to describe the
internal shape
of these interconnected structures, in which there are many statistics of interest, e.g., enumeration, smoothness, component sizes, limit shapes, large deviations, etc.  The underlying theme is quantifying the severity of dependence between the components with respect to any desired statistic.  

A motivating example is the set of integer partitions of size~$n$ into exactly $k$ parts, for which there is an extensive history. 
Integer partitions are an example of a multiset, and thus are not covered in our scope.  
In addition, specifically for integer partitions there is a well-known bijection which allows one to instead consider the number of partitions of $n$ into parts of size at most $k$, and the generating function is most obliging in this setting for asymptotic analysis~\cite{ErdosLehner, Szekeres1, Szekeres2}. 
Erd\H{o}s and Lehner looked at $k \sim \frac{\sqrt{n}}{2c} \log n$, which governs the size of the largest part in a uniformly chosen integer partition of size~$n$. 
For $k$ small, i.e., $k = o(\sqrt{n})$, there is one asymptotic formula for $p(n,k)$ that holds uniformly~\cite{Szekeres1}. 
For larger $k$, if one is content with formulas involving implicit parameters which change character depending on $k$, then~\cite[Theorem~1]{Szekeres2} provides a complete answer, which includes the Hardy-Ramanujan formula as a special case; this was later proved by elementary means in~\cite{CanfieldElementary} and through probabilistic means in~\cite{Romik}. 
One of the motivations for obtaining such enumerative formulas is presented in~\cite[Theorem~2]{Szekeres2}, which demonstrates that the number of partitions of size~$n$ into exactly $k$ parts is unimodal in $k$. 
Similar treatments were undertaken for partitions into distinct parts~\cite{AlmkvistUnimodal}, and indeed unimodality and log-concavity is a traditional and lively topic in enumerative combinatorics~\cite{StanleyLogConcave, BrentiUpdate}.
From a probabilistic point of view, such smoothness properties of a sequence correspond to central and local limit theorems, see~\cite{Bender, LCLT2, LCLT3}.

A closer parallel, and indeed a special case which motivated the present document, is the asymptotic analysis surrounding Stirling numbers of the second kind, denoted by $S(n,k)$, 
corresponding to the basic combinatorial assembly, set partitions.
There is an equally extensive history for Stirling numbers, see for example~\cite{Louchard2, ChelluriRichmond, Stirling} and the references therein, which begins with the elementary analysis of Jordan~\cite{Jordan}  which covers $S(n,k)$ and $S(n,n-r)$ for $k$ and $r$ fixed. 
As with integer partitions, one must decide what constitutes an answer, as the seminal work of Moser and Wyman~\cite{MoserWyman2} essentially characterizes the asymptotic behavior for the entire range of values of $k$, \emph{albeit with implicitly defined parameters}. 
A detailed analysis was recently carried out by Louchard~\cite{Louchard2}, in which he obtained an asymptotic expansion in terms of explicit parameters, but with gaps at $r \sim t\, n^{\ell/(1+\ell)}$ for $t>0$ and $\ell = 1,2,\ldots$. 
The gap at $\ell=1$ was filled by the authors in~\cite{Stirling}, using a bijection involving non-attacking rooks on a lower triangular 
chess board. 
The intuition is that while saddle point analysis paints a broad brush, it is essentially capturing a central limit theorem, dominated by a normal distribution; however, at $r \sim t\, n^{\ell/(1+\ell)}$, the behavior is Poissonian, and so the usual analytical methods either break down or become obfuscated by the implicitly defined parameters. 
Stirling numbers of the first kind follow a similar tradition, see for example~\cite{Jordan, MoserWyman1, Louchard1, Stirling, ChelluriRichmond} and the references therein.

Another important aspect of asymptotic analysis is the distinction between obtaining big O error bounds, with statements such as \emph{there exists an $n_0$ such that \ldots}, versus quantitative bounds, which provide statements such as \emph{for all $n \geq 26$ \ldots}. 
For example, one can prove using the first term of the Hardy-Ramanujan asymptotic formula~\cite[Equation~(5.5)]{HR18} that \emph{there exists an $n_0$ such that} the number of integer partitions of $n$ is a log-concave sequence for all $n \geq n_0$, but their analysis cannot specify the smallest value of $n$ for which this property is guaranteed to hold\footnote{Note that one should not attempt to prove even the asymptotic log-concavity of the partition function using~\cite[Equation~(1.41)]{HR18}; see~\cite[Section~7]{logconcave} and~\url{http://tinyurl.com/kkc6fwf}.}. 
In order to obtain $n_0 = 26$, one would need the quantitative bounds provided by Rademacher~\cite{Rademacher} or Lehmer~\cite{Lehmer39}, along with some elementary albeit tedious analysis, see~\cite{Nicolas}; see also~\cite{logconcave}.

\subsection{Decomposable combinatorial structures}
\emph{Decomposable} combinatorial structures of size $n$  are often examined with respect to the underlying integer partition of $n$.  For a given instance of the structure of size $n$, having $k$ components, the integer partition is written $\lambda = (\lambda_1,\lambda_2,\dots,\lambda_k)$.  Here  $\lambda_j$ is the size of the $j$th largest component,  so that $n=\lambda_1+\cdots+\lambda_k$, and $\lambda_1 \ge \lambda_2 \ge \cdots \ge \lambda_k \ge 1$.
Picking uniformly at random from the $p(n)$ structures of size $n$,  the corresponding random integer partition is $(L_1,L_2,\ldots,L_K)$,  where we write  $K \equiv K_n$ for the number of components,  and $L_j \equiv L_j(n)$ is defined to be the size of the $j$th largest component,  with the added provision that $L_j=0$ if $j>K$,  so that the random variable $L_j$ is defined for all $j \in \BN$, the set of positive integers.   We write  $\BZ_+$ for the set of non-negative integers.

An alternate way to report the information in $L_1,L_2,\ldots$ is to consider  $C_i \equiv C_i(n)$, the number of components of size $i$; of course $C_i(n):=0$ if $i>n$.  The notation $C_i \equiv C_i(n)$ says that we consider both $C_i$ and $C_i(n)$ to be the same;  which notation gets used depends on whether or not one wishes to emphasize the role of the parameter $n$.
The entire process of component counts is
\begin{equation}\label{def C}
  \bC  \equiv \bC(n) := (C_1(n),C_2(n),\ldots,C_n(n)).
\end{equation}
To review, the following two identities are trivial:
$C_1+2 C_2 +\cdots =n,  \ \  C_1+C_2 +\cdots = K$.  

{\bf I}ndependent {\bf p}rocess {\bf a}pproximations for {\bf r}andom {\bf c}ombinatorial {\bf s}tructures, 
conveniently abbreviated as IPARCS,  would be, in the greatest generality, a choice of decomposable combinatorial structure, and independent non-negative integer-valued random variables $Z_1,Z_2,\ldots$
such that, with respect to \emph{some} functionals $\phi: \BZ_+^n \to \BR$ of interest, the distribution of 
$\phi(\bC(n))$ is well-approximated by the distribution of $\phi( (Z_1,\ldots,Z_n) )$.   Examples of natural functionals of interest include:

\begin{itemize}
\item  The least common multiple of all component sizes, $\text{lcm}(L_1,\ldots,L_k)$.
\item  The indicator, 0 or 1, of the statement that all component sizes are distinct.
\item The size $L_1$ of the largest compnent.
\item The difference  $L_1-L_2$,  or the ratio  $L_2/L_1$,  for the largest and second largest component.

\item The number of components,  $K$ versus its approximation $Z_1+\cdots+Z_n$.
\item The number of components of size at most $b(n)$, versus $Z_1+\cdots+Z_b$, for some given function $b: \BN \to \BN$.
\item The number of components of size at least $a(n)$, versus $Z_a+\cdots+Z_n$,  for some given function $a: \BN \to \BN$.
\item The number of components of size in the range $[a(n),b(n)]$.
 
\item The \emph{process} of all small components,  $(C_1,\ldots,C_{b(n)})$, versus $(Z_1,\ldots,Z_b)$, for some given function $b: \BN \to \BN$.
\item The \emph{process} of all large components,  $(C_{a(n)},\ldots,C_n)$, versus $(Z_a,\ldots,Z_n)$, for some given function $a: \BN \to \BN$.

\end{itemize}
  
For each of the above functionals, there are examples of fundamental and natural combinatorial objects where a sensibly chosen independent process gives a good approximation, with respect to that functional, and there are other combinatorial examples where a sensibly chosen independent process does not give a good approximation.  Good examples for the last two functionals listed above are given by Pittel
\cite{PittelShape,PittelSetPartitions}, for integer partitions and set partitions, respectively.  
For contrast, an example of a functional where one \emph{never}
gets a good approximation:   the functional is the weighted sum of the component counts, so that
$\phi(\bC(n)) = C_1(n)+2C_2(n)+\cdots+nC_n(n)$ is the constant random variable of value $n$, while $\phi((Z_1,Z_2,\ldots,Z_n))=Z_1+2 Z_2+\cdots+n Z_n$ is not constant.

\subsection{The conditioning relation} \label{sect CR}

In addition to having independent random variables $Z_1,Z_2,\ldots,Z_n$ which give a usable approximation to a combinatorial structure $\bC(n)$, it is often the case, as surveyed in \cite{IPARCS} and \cite{ABTbook},  that there is a one parameter family of distributions, indexed by $x>0$ or $x$ in a bounded range such as $(0,1)$,
such that for every choice of $x$,  there is an equality of joint distributions, after conditioning on the event that $T_n=n$,  where  
\begin{equation}\label{T_n}
T_n := Z_1+2Z_2 + \cdots + n Z_n;
\end{equation}
that is, 
\begin{equation}\label{CR}
\text{ {\bf Conditioning Relation}:}  \ \ \ \cL(\bC(n)) = \cL_x((Z_1,Z_2,\ldots,Z_n) \ | \ T_n=n).
\end{equation}
When the value of the parameter $x$ has been chosen,  the above might be denoted more simply as $\bC(n) =^d (Z_1,Z_2,\ldots,Z_n) \ | \ T_n=n)$.  Implicit in the conditioning relation and the combinatorial setup is a relation, which in the case of \emph{combinatorial assemblies} 
is \eqref{x identity}, expressing $p(n)$ exactly, in terms of the normalizing constants for the random variables $Z_i$ -- these vary with $x$,  and $\p_x(T_n=n)$, which also varies with $x$.
 
In the context of probability theory, the \emph{saddle point heuristic} \cite{Daniels,saddlebook,Reid}  is that probability approximations, such as those recognizable by the factor $1/\sqrt{2 \pi \sigma^2}$ from the central limit theorem,  tend to be more accurate when appropriately tilting, perhaps by the Cramer tilt (called the Esscher tilt in \cite{saddlebook}, after \cite{esscher}).   Thanks to \eqref{x identity},
we can recognize   saddlepoint approximations to $p(n)$, the number of size $n$ instances of  a given type of assembly, as corresponding to approximations for $\p_x(T_n=n)$ carried out by picking a value of the parameter $x$ for which $\ex T_n$ is near $n$.

The \emph{extended} saddle point heuristic is in one sense more specific:  it says that in the $x$-indexed family of distributions for the right side of \eqref{CR},  when one wants to \emph{remove}
the conditioning on $T_n=n$ and use the independent process $(Z_1,Z_2,\ldots,Z_n)$ as an approximation
(with respect to various functionals $\phi$) for $\bC(n)$, good approximations are obtained by picking the parameter $x$  for which  $\ex T_n$ is equal to, or close to, its target in the conditioning,
$n$.  

The main point of the present paper is to extend this extended saddle point heuristic  to combinatorial assemblies of size $n$ and having $k$ components,  where the given size  $k=k(n)$ may be far from the typical number of components of a random structure of size $n$.  We provide a construction, \eqref{Y} -- \eqref{CRY T}, which forces the number of components to be a given $k$,  and still leaves a parameter $x$ and a sequence of independent random variables $Y_1,\ldots,Y_k$,  such that \emph{conditional} on the event $Y_1+\cdots+Y_k=n$, we achieve \emph{exactly} the distribution of the component structure $\bD(n,k)$ of a random assembly of size $n$ with $k$ components, with all $p(n,k)$ possibilities equally likely.  As a heuristic, when the parameter $x$ is chosen so that $\e(Y_1+\cdots+Y_k)$ is close to $n$,  the independent sample $Y_1,\ldots,Y_k$ is close in distribution, with respect to various functionals $\phi$,  to the distribution of $\bD(n,k)$.

\subsection{Three classes of examples:  selections, multisets, and assemblies}
\label{sect examples}

We take these three main classes in order of the superficial complexity of description; in particular,  assemblies come last.

\emph{Selections and multisets}   To begin,  one assumes that there is a universe $U$ of objects having a positive integer-valued weight,  such that for $i=1,2,\ldots$, the number $m_i$ of objects of weight $i$ satisfies:  $m_i$ is finite.  

 \emph{Selections}, in this context,  are simply all finite subsets of $U$,  with the natural notion, that the weight of a set is the sum of the weights of its elements.  Let $p(n)$ be the number of sets of weight $n$;  always,
$p(0)=1$, with the emptyset being the unique set of weight 1.   The characterizing ordinary generating function for this story is
\begin{equation}\label{selection p}
     P(z) := \sum_{n \ge 0} p(n) z(n) = \prod_{i \ge 1}  (1+x^i)^{m_i}.
\end{equation}
The corresponding independent random variables for the conditioning relation \eqref{CR} are Binomial($m_i,x^i/(1+x^i))$,  so that the distribution of $Z_i$ is the $m_i$-fold convolution of the Bernoulli($p=x^i/(1+x^i))$ distribution on $\{0,1\}$.

 \emph{Multisets}, in this context,  are all finite cardinality multisubsets of $U$,  with the natural notion, that the weight of a multiset is the sum of the weights of its elements.  Let $p(n)$ be the number of multisets of weight $n$;  always,
$p(0)=1$, with the emptyset being the unique multiset of weight 1.   The characterizing ordinary generating function for this story is
\begin{equation}\label{multiset p}
     P(z) := \sum_{n \ge 0} p(n) z(n) = \prod_{i \ge 1}  (1-x^i)^{-m_i}.
\end{equation}
The corresponding independent random variables for the conditioning relation \eqref{CR} are NegativeBinomial($m_i,x^i)$,  so that the distribution of $Z_i$ is the $m_i$-fold convolution of the Geometric  (starting from zero, with ratio $x^i$) distribution on $\{0,1,2,\ldots\}$, that is, the distribution of $G$ with $\p(G \ge k)=x^{ik}, k=0,1,2,\ldots$.

Examples of multisets and selections include

\begin{itemize}

\item  {\bf Integer partitions}   Here $m_i=1$ for all $i \ge 1$,  and the sole object of weight $i$ in the universe 
$U$ \emph{is} the integer $i$ itself.   Multisets correspond to integer partitions, with no restrictions, and our $p(n)$ is the usual $p_n$,  satisfying the asymptotic relation 
\begin{equation}\label{HR}
p_n \sim \frac{e^{2 \sqrt{n \pi^2/6}}}{\sqrt{48}n}, 
\end{equation}
as found, and also more effectively approximated, by    Hardy, Ramanujan, Rademacher, and Lehmer,  \cite{HR18,Rademacher,Lehmer38,Lehmer39}.  Selections correspond to  integer partitions with all parts distinct.

\item{Polynomials over $\BF_q$}.  Here, the universe $U$ is the set of monic  irreducible polynomials over the finite field $\BF_q$ with $q$ elements, with weight being degree.  By unique factorization, multisets correspond to monic polynomials, and putting $p(n)=q^n$ 
in \eqref{multiset p} leads to the relations $q^n = \sum_{i | n}
i m_i$,  with $m_i \equiv m_i(q)$, which by M\"obius inversion is equivalent to $i m_i =\sum_{d | i} \mu(i/d) q^d$, so that
\begin{equation}\label{polys}
m_i(q) = \frac{1}{i}   \sum_{d | i} \mu(i/d) q^d   \sim \frac{q^i}{i};
\end{equation}
  this was known to Gauss, see \cite{Berlekamp}.  The explicit formula shows that as $i \to \infty$, $m_i(q) = (q^i/i)(1+O(q^{-i/2}))$ which is a remarkably easy and effective analog of the prime number theorem.   While multisets correspond to all monic polynomials, \emph{selections} correspond to \emph{square-free} monic polynomials.  The generating function \eqref{multiset p}
  does not converge at $x=1/q$,  but $x=1/q$ is the correct value to use,  so that the independent
  $Z_i$ have  $\e_x  Z_i  \sim 1/i$ as $i \to \infty$,  and we have a logarithmic combinatorial structure in the sense of \cite{ABTbook}, with the number of irreducible factors of a random polynomial of degree $n$  growing like $\log n$.

\item {Necklaces over an alphabet of size $q$.}  Same as the above, without the restriction that $q$ is a prime power; see for example~\cite{vanLint}.

\end{itemize}

\emph{Assemblies}  appear in 1974, called abelian partitioned composite, and recognizable via the exponential relation \eqref{exponential relation}, in Foata \cite{Foata}.  Assemblies are called \emph{species} by Joyal \cite{Joyal,Joyal2}; they are called exponential families, with decks and hands, by Wilf \cite{Wilf}; they are called \emph{uniform structures} in \cite[Theorem 14.2]{vanLint}; and   they are discussed extensively as the \emph{SET} construction for labelled structures in \cite[Section II]{Flajolet}. 
Wilf \cite[Section 3.18]{Wilf} supplies the early history:  Riddell in a 1951 thesis \cite[Footnote 18]{RU} has the exponential formula 
\eqref{exponential relation} in the context of simple graphs.  Bender and Goldman in 1971 \cite{BG}  and Foata and Sch\"utzenberger in 1970 \cite{FS} have the general assembly, called \emph{prefab} in the former, and \emph{compos\'e partitionelle} in the latter.

The simplest assembly  is \emph{set partition},  equivalently, arbitrary
\emph{equivalence relation}:  for a set partition of size $n$, the set  $[n]:=\{1,2,\ldots,n\}$ is decomposed as a disjoint union of nonempty subsets, referred to as the \emph{blocks}  or \emph{equivalence classes},  and the blocks are gathered as a set of blocks, rather than a list of blocks.   For an integer partition of $n$
having counts $\ba$, that is,  $a_i$ parts of size $i$, $i=1$ to $n$,  so that
\begin{equation}\label{def a}
\ba = (a_1,a_2,\ldots, a_n) \in \BZ_+^n, \text{ with }  a_1 + 2a_2 + \cdots + n a_n=n,
\end{equation}
the number of set partitions of \emph{type} $\ba$, that is, having $a_i$ blocks of size $i$, $i=1$ to $n$, is
\begin{equation}\label{set partition count}
N(n,\ba) = n! \ \frac{1}{\prod_1^n a_i! (i!)^{a_i}}.
\end{equation}

The general assembly is specified by a sequence of nonnegative integers $m_1,m_2,\ldots$,  which is encoded in the exponential generating function
\begin{equation}\label{def M}
M(z) := \sum_{i \ge 1}  \frac{m_i \, z^i}{i!}.
\end{equation}
If one doesn't explicitly state the range of summation, $i \ge 1$, then one should specify that $m_0=0$.
An instance of this assembly, of size $n$, is formed in two steps:  1)  pick a set partition on $[n]$, and 2)  for each block of size $i$,  \emph{decorate} that block in one of $m_i$ ways.  Hence, the
the number of $M$-assemblies  of \emph{type} $\ba$ is
\begin{equation}\label{assembly count}
N(n,\ba) = n! \ \prod_1^n  \frac{m_i^{a_i}}{ a_i! (i!)^{a_i}}.
\end{equation}
Summing over all $\ba$ satisfying \eqref{def a} yields $p(n)$, the total number of $M$-assemblies of size $n$; wrapping these up in an exponential generating function yields 
\begin{equation}\label{assembly p}
P(z):= \sum_{n \ge 0} p(n) z^n / n! .
\end{equation}   
Note that we always have $p(0)=1$ for the trivial but confusing reason that there is a unique equivalence relation on the empty set,  and since this has no blocks, there is exactly one way to decorate all these blocks.

The succinct characterization of an $M$-assembly is the exponential relation on exponential generating functions,
\begin{equation}\label{exponential relation}
     P(z) = \exp(M(z)).
\end{equation}
This is always a valid relation, in the sense of formal power series;  as an exercise, the reader is urged to name the radius of convergence of each of the examples presented after 
\eqref{x identity}.     
     
The corresponding independent random variables for the conditioning relation \eqref{CR} are Poisson, with
\begin{equation}\label{lambda}
 \ex Z_i  = \lambda_i \equiv \lambda_i(x) := \frac{m_i \, x^i}{i!},
\end{equation}
valid for any $x>0$.  The identity \eqref{CR} is very easily proved,  by comparison with \eqref{set partition count},  and by equating the normalizing constants for $\p(\bC(n)=\ba)$ and $\p_x((Z_1,Z_2,\ldots,Z_n) = \ba)$,  where $\ba$ as in \eqref{def a} has weighted sum $a_1+2a_2+\cdots+n a_n=n$,  whence we get the \emph{identity}
\begin{equation}\label{x identity}
     p(n) = \frac{n!}{x^n} \   \exp\left(\lambda_1(x)+\cdots+\lambda_n(x)\right)  \ \p_x(T_n=n).
\end{equation}

Examples of assemblies include

\begin{itemize}
\item  {\bf Set partitions}, decomposed into blocks.  Here,  $m_i=1$ for all $i \ge 1$, so $M(z)=e^z-1$,  and $p(n)$ is usually denoted as $B_n$, the $n$th Bell number.

\item  {\bf Permutations}, decomposed into cycles.   Here,  $m_1=m_2=1,m_3=2$; in general $m_i=(i-1)!$.  We have   $M(z)=\log(1-z)$, $p(n)=n!$, and $P(z)= 1/(1-z)$.

\item {\bf Random mappings}, i.e., arbitrary functions  $f:[n]  \to [n]$.   Of course, $p(n)=n^n$. The components are the
weakly connected components of the {random mapping digraph},  i.e. the  directed graph exactly $n$ edges, namely $(i,f(i))$.   Here $m_1=1,m_2=2, m_3=17$, and it turns out that $m_i = (i-1)! \sum_{j=0}^{i-1}  i^j/j!$.  

\item {\bf Simple graphs}, i.e., undirected graphs with no loops and no multiple edges.  Here, $p(n)=2^{n \choose 2}$,   and for small $i$, $m_i$ can be computed from the relation $P(z)=e^{M(z)}$.
It is easy to show that as $i \to \infty$, $m_i \sim p(i)$.  Hence  $M(\cdot)$ has radius of convergence zero,  and does \emph{not} satisfy the hypotheses of Theorems \ref{thm 1}.   This assembly does not follow the behavior described by Theorems \ref{thm 1}: picking uniformly from the $p(n,k)$ simple graphs on $n$ vertices with exactly $k$ components,  for $k= n -r$ where $r= \lfloor \sqrt{n} \rfloor$, it is fairly easy to see that  with probability tending to 1,  there is one large component, of size $r+1$,  and the other $n-r-1$  components are singletons.

\end{itemize}

\subsection{The finite case,  $T_n$, versus the infinite case,  $T$} \label{sect T_n versus T}

In some situations involving the conditioning relation \eqref{CR} and the finite sum $T_n$ in \eqref{T_n},  the infinite sum
\begin{equation}\label{T}
T := Z_1+2Z_2 + \cdots 
\end{equation}
is more convenient to work with.  The random variable $T$ takes values in the extended nonnegative integers, $\{0,1,2,\ldots,\infty\}$,  and one of the requirements for $T$ to be useful is that $\p_x (T < \infty)=1$.
When $T$ is to be used,  recalling that $C_i(n):=0$ whenever $i>n$, 
the conditioning relation \eqref{CR} is changed to
\begin{equation}\label{CR infinite}
\text{ {\bf CR}:}  \ \ \ \cL((C_1(n),C_2(n),\ldots) = \cL_x((Z_1,Z_2,\ldots) \ | \ T=n).
\end{equation}
(If one is not using the finite sum $T_n$,  but only the infinite sum $T$, it would make sense to reuse
 the notation from \eqref{def C}, and define $\bC \equiv \bC(n):= (C_1(n),C_2(n)\ldots)$,  but since the purpose of this section is to clarify the similarities and differences between the two setups, we don't bother giving  $(C_1(n),C_2(n)\ldots)$ its own symbol.)
Note that \eqref{CR} is valid,  in the example where the combinatorial structure is  simple graphs, described at the end of Section \ref{sect examples},  even though the parameter $x$ is necessarily outside the circle of convergence, and $\p_x(T=\infty)=1$.  But this is \emph{not} an informative example for the choice $T_n$ versus $T$,  since for this example, the extended saddle point heuristic, described at the end of Section \ref{sect CR}, does not provide useful approximation for \emph{any} value of the parameter $x$.   

An informative example is random permutations. The exponential generating function in \eqref{assembly p} is $P(z)=1/(1-z)$, so that at $z=1$,  the series diverges.   Shepp and Lloyd \cite{Shepp} consider the conditioning relation \eqref{CR infinite} involving  $T$, with parameter $x<1$, so that $\e Z_i = x^i/i$ and $\e T = 1/(1-x)<\infty$, and take  but $x \to 1$ to get results by applying a Tauberian theorem. In contrast,  \cite[Theorem 2]{AT92} uses elementary analysis to get a completely effective error bound, with superexponentially fast decay, by considering $x=1$ 
but using $T_n$, with $\ex T_n =n$  at $x=1$.

Another informative example is polynomials over $\BF_q$.  The ordinary generating function in \eqref{multiset p} is $P(z) = \sum_{n \ge 0} p(n) z^n =  \sum_{n \ge 0} q^n z^n  = 1/(1-qz)$, and again, the most useful parameter choice is $x=1/q$, on the boundary of the region of convergence, but where the series converges and $1=\p(T=\infty)$.  One can work with the choice $x=1/q$ and the conditioning relation, as in \cite{ABTbook},  or more directly get approximations involving the $Z_i$, which are  NegativeBinomial($m_i,x^i/(1+x^i))$, with $x=1/q$, using inclusion-exclusion, as in \cite{ABT93}.

\ignore{  formal power series   versus analytic function,  T_n  versus T,   IPARCS and ABTbook use T_n,   Shepp and Lloyd used T but AT92   used T_n,   Flajolet book uses T.
and ? Fristead used ?

\cite{Flajolet}
\cite{AT92}
\cite{ABT93}
}

\section{Choosing uniformly from the $p(n,k)$ objects with $k$ components}\label{sect n k}

For a given type of decomposable combinatorial object, let $p(n,k)$ be the number of instances of size $n$, having exactly $k$ components.   For $M$-assemblies,  $p(n,k)$ can be computed from
\eqref{assembly count} by summing $N(n,\ba)$ over all $\ba$ with \emph{both}  $a_1+2a_2+\cdots+na_n=n$ \emph{and} $a_1+a_2+\cdots+a_n=k$.    More succinctly, $p(\cdot,\cdot)$ is determined by the two variable generating function relation
\begin{equation}\label{theta x}
  \sum_{n,k \ge 0} p(n,k) \frac{z^n \theta^k}{n!} = e^{\theta M(z)},
\end{equation}
which reduces to \eqref{exponential relation} by setting $\theta=1$.  

For set partitions,  $p(n,k)$ is typically denoted $S(n,k)$,  and is called a Stirling number of the second kind.  For permutations,  $p(n,k)$ is typically denoted $s(n,k)$ or $|s(n,k)|$, and is called an (unsigned) Stirling number of the first kind.

Pick uniformly from the $p(n,k)$ instances of a structure of size $n$ with $k$ components.  Similar to \eqref{def C},  we write $D_i \equiv D_i(n,k)$ for the number of components of size $i$.
The entire process of component counts is
\begin{equation}\label{def D}
  \bD  \equiv \bD(n,k) = (D_1(n,k),D_2(n,k),\ldots,D_n(n,k)).
\end{equation}
To review, the following two identities are trivial:
$D_1+2 D_2 +\cdots =n,  \ \  D_1+D_2 +\cdots = k$.  
Recall that in the similar statement following \eqref{def C},   we wrote  $C_1+C_2+\cdots=K$,  with the uppercase $K$ being the random number of components of a uniformly chosen instance of size $n$;  here for $D_1+D_2+\cdots$ we have the lowercase $k$, which is constant (even though for applications one takes $n,k \to \infty$ together, which is usually described via $k=k(n)$.)

Trivially, picking uniformly from a finite set $B$ (all $p(n)$ assemblies of size $n$) and then conditioning on landing in a given subset $B_0$  (all $p(n,k)$ assemblies of size $n$ with $k$ components) yields the uniform distribution on the subset $B_0$. Hence
for every $n,k$,  the distribution of $\bD(n,k)$ is the condtional distribution of $\bC(n)$ given that $K=k$:
\begin{equation}\label{trivial}
  \bD(n,k)  =^d  ( \bC(n) \ | \ K=k).
\end{equation}
    
It  then follows from \eqref{CR} combined with \eqref{trivial} that for every $x>0$
\begin{equation}\label{CR with k T_n}
\text{ {\bf CR}:}  \ \ \ \cL(\bD(n,k)) = \cL_x((Z_1,Z_2,\ldots,Z_n) \ | \ T_n=n \text{ and } Z_1+\cdots+Z_n=k).
\end{equation}    
(This may be confusing because of the change in notation, from $K$ to $Z_1+\cdots+Z_n$, but it is genuinely trivial, and corresponds to the associative property of multiplication:  we are biasing the distribution of $(Z_1,\ldots,Z_n)$ first by multiplying in the indicator of the event $Z_1+2Z_2+\cdots+nZ_n=n$, and then multiplying in the indicator of the event $Z_1+Z_2+\cdots+Z_n=k$.)

Finally, as in Section \ref{sect T_n versus T}, one may also work with the infinite sum $T$ from \eqref{T},  to get
\begin{equation}\label{CR with k T}
\text{ {\bf CR}:}  \ \ \ \cL((D_1(n,k),D_2(n,k),\ldots) = \cL_x((Z_1,Z_2,\ldots) \ | \ T=n \text{ and } Z_1+Z_2+\cdots=k).
\end{equation}

\subsection{Poisson process, conditional on having $k$ arrivals}\label{sect Poisson}

Consider a general Poisson process on a space $S$, with intensity measure $\mu$;  for simplicity of exposition, we restrict to the case $\lambda := \mu(S)< \infty$.  See for example 
\cite[Section II.37]{Williams}.  (A note on notation: in \eqref{lambda} -- \eqref{x identity}, and in this section,  $\lambda$ and $\lambda_i$ serve as Poisson parameters, as is typical notation in standard probability texts; in contrast, in the sections highlighting combinatorial arguments, $\lambda_i$ denotes the $i$th part of a partition $\lambda$ of the integer $n$, as is standard notation in combinatorics texts.) The Poisson  process is characterized by the requirement that for disjoint (measurable) $B_1,B_2,\ldots,B_r \subset S$,  with $N_i \equiv N(B_i)$ defined as the number of arrivals in $B_i$, one has that $N_i$ is Poisson with parameter $\mu(B_i)$, and   $N_1,\ldots,N_r$ are mutually independent.   Given $\mu$, there is a very simple construction of the desired Poisson process:  Let $Y$ be a random element of $S$, with distribution $(1/\lambda) \mu$, and let $Y,Y_1,Y_2,\ldots$ be i.i.d., and independent of a random variable $Z$, taken to be Poisson with  parameter $\lambda$. Now one simply defines the (multiset) of all arrivals to be the sample of random size $Z$, i.e., the multiset $\{Y_1,Y_2,\ldots,Y_Z\}$, so that  for any measurable $B \subset S$, $N(B) = \sum_{i \ge 1} 1(Y_i \in B, i \le Z)$.  (A similar story holds in the case where  $\mu$ is a sigma-finite meaure on $S$ with $\mu(S)=\infty$,  but one has to restrict to regions $R \subset S$ with $\mu(R)<\infty$, and some care must be taken to put the sigma-finite pieces together.)

The result of the above-described coupling is that, conditional on having $k$ arrivals overall, the arrivals \emph{are} the i.i.d.~sample of size $k$, i.e., $\{Y_1,Y_2,\ldots,Y_k\}$ considered as a multisubset of $S$. 
For the  purpose of approximating the component structure of assemblies,  one takes  $S=[n]$ if the intention is to use \eqref{CR},  and $S=\BN $  if the intention is to use \eqref{CR infinite}.
In either case, to write out the explicit recipe,  
recall \eqref{lambda}, that
$\lambda_i \equiv \lambda_i(x) := m_i \, x^i / i!$,
and let
\begin{equation}\label{Y}
\p_x(Y=i) = \left\{  
\begin{array}{cc}
     \frac{\lambda_i}{\lambda_1+\cdots + \lambda_n}  & \text{if } i \in S=[n] \\
     \frac{\lambda_i}{\lambda_1+\lambda_2+\cdots}=\frac{\lambda_i(x)}{M(x)}  & \text{if } i \in S= \BN, \text{ and } M(x)<\infty.
 \end{array}
 \right.
 \end{equation}
 Let $Y,Y_1,Y_2,\ldots$ be i.i.d., and independent of a random variable $Z$, taken to be Poisson with  parameter $\lambda$,
 with  $\lambda = \lambda_1+\cdots + \lambda_n$ for the case $S=[n]$,  and  $\lambda = \lambda_1+\lambda_2+\cdots = M(x)$, for the case $S=\BN$ and $M(x)< \infty$.   We use these to construct, simultaneously,  the Poisson process, and for each $k=0,1,2,\ldots$,
a realization of the Poisson process conditional on having $k$ arrivals overall.

We view the above construction as a \emph{coupling}.  Given a value $k \ge 0$, we write $N_i \equiv N_i(k)$ for  the count of how many of $Y_1,\ldots,Y_k$ are equal to $i$.  Hence  $N_1+\cdots+N_n=k$
in case $S=[n]$,  or  $N_1+N_2+\cdots=k$, in case $S=\BN$.  We have, for each $x>0$
$$
   (N_1,\ldots,N_n)  =^d  ( (Z_1,\ldots,Z_n) | Z_1+\cdots+Z_n = k),   \text{  if }  S=[n] 
$$
and 
 $$
   (N_1,N_2,\ldots)  =^d  ( (Z_1,Z_2,\ldots) | Z_1+Z_2+\cdots=k),   \text{ if }  S=\BN, \text{ and } M(x)<\infty.
$$
Note that the sum of the arrivals $Y_j$ is the weighted sum of the counts, i.e., in the case $S=[n]$, always $Y_1+\cdots + Y_Z =N_1(Z)+2N_2(Z)+\cdots+n N_n(Z)$.  
Hence by further conditioning on the weighted sum of the $Z_i$ being equal 
to $n$, \eqref{CR with k T_n}  and  \eqref{CR with k T}  imply that
\begin{equation}\label{CRY T_n}
\text{ {\bf CR}:}  \ \ \ \cL(\bD(n,k)) = \cL_x((N_1,N_2,\ldots,N_n) \ | \ Y_1+\cdots+Y_k = n).
\end{equation}    
\begin{equation}\label{CRY T}
\text{ {\bf CR}:}  \ \ \ \cL((D_1(n,k),D_2(n,k),\ldots)) = \cL_x((N_1,N_2,\ldots) \ | \  Y_1+\cdots+Y_k = n).
\end{equation}   

Each of \eqref{CRY T_n} and \eqref{CRY T} leads to an exact expression, similar to 
\eqref{x identity} for $p(n)$, but now  relating $p(n,k)$ and
$\p_x(Y_1+\cdots+Y_k=n)$.
Let $\ba$ satisfy \eqref{def a} and the additional condition that $a_1+a_2+\cdots+a_n=k$.
Then  \eqref{assembly count} implies that
\begin{equation}\label{k identity}
\p(\bD(n,k) = \ba)=\frac{n!}{p(n,k)} \  \prod_1^n  \frac{m_i^{a_i}}{ a_i! (i!)^{a_i}}.
\end{equation}
On the right side of \eqref{CRY T_n} and \eqref{CRY T},  the conditional probability has
denominator  $\p_x(Y_1+\cdots+Y_k=n)$  and numerator
\begin{equation}\label{multinomial}
\p_x((N_1,N_2,\ldots)=\ba)= {k \choose a_1,a_2,\ldots,a_n} \ \prod_1^n \ (\p_x(Y=i))^{a_i}.
\end{equation}
Combining \eqref{lambda}; the first case of \eqref{Y}; \eqref{CRY T_n} as the statement that
$\p(\bD(n,k) = \ba) =\p_x((N_1,N_2,\ldots,N_n)=\ba)/\p_x(Y_1+\cdots+Y_k=n)$; \eqref{k identity};
\eqref{multinomial}; and cancelling common factors, we get
\begin{equation}\label{pnk T_n}
\forall x>0, \ p(n,k) =\frac{n!}{k!} \  \frac{\left(\lambda_1(x)+\cdots+\lambda_n(x)\right)^k}{x^n} \ \p_x(Y_1+\cdots+Y_k=n).
\end{equation}
Using instead the second case of \eqref{Y}, and \eqref{CRY T}, in a similar way we get
\begin{equation}\label{pnk T}
\forall x\!:\! M(x)<\infty,\  \ p(n,k) =\frac{n!}{k!} \  \frac{\left(M(x)\right)^k}{x^n} \ \p_x(Y_1+\cdots+Y_k=n).
\end{equation}

Note the for the above,  so long as $a_1+2a_2+\cdots+n a_n=n$ and $a_1+a_2+\cdots+a_n=k$, 
all factors depending on $\ba$ cancel, leading to \eqref{pnk T_n} and \eqref{pnk T},  which are 
identities in $x$, with no trace of $\ba$ left behind.   Perhaps it will come as a surprise that there is another strategy for using estimates of
$\p_x(Y_1+\cdots+Y_k=n)$  to give estimates for $p(n,k)$,  for which the key is to name a \emph{specific} pivotal choice of $\ba$.   We carry this out in Corollary \ref{cor 1}, to get
asymptotics for $p(n,k)$ when $n-k \sim t \sqrt{n}$.

\subsection{Two versions of the $k$-Boltzmann Sampler}

The idea of ``Boltzmann sampling", popularized by \cite{Boltzmann},  is that when one wants to sample a structure of a given size $n$, uniformly distributed over the $p(n)$ possibilities of that size, it may be useful to ignore the requirement of getting size exactly $n$,  and instead generate a random object of size $T_n$ having mean around $n$.   From our point of view,  this is simply the 
combination of 
\eqref{T_n} and \eqref{CR}  while not requiring the occurrence of the conditioning event $\{T_n=n\}$, together with the extended saddle point heuristic,  that the conditioning has a mild effect on many functionals
of the joint distribution of $(Z_1,Z_2,\ldots,Z_n)$  when the tilting parameter $x$ is chosen so that
$\ex T_n$ is close to the target $n$.   One virtue of \cite{Boltzmann} relative to \cite{IPARCS} is that the classes studied include more than just assemblies, multisets, and selections.

Now suppose  one wants,  given $n$ and $k$,  to sample a structure, uniformly distributed over the $p(n,k)$ possibilities of size  $n$ having exactly $k$ components.   For the special case of assemblies,  the Poisson-process-inspired conditioning relations,  \eqref{CRY T_n} and 
\eqref{CRY T},  provide a very convenient analog of the Boltzmann sampler.  This analog is to pick the parameter $x$ for the distribution of $Y$ in \eqref{Y}  so that  $k\,  \ex Y$ is close to the target $n$,  then take an i.i.d.  sample $Y_1,\ldots,Y_k$ from this distribution,  and ignore the requirement
that $Y_1+\cdots+Y_k=n$.   In effect,  we generate a random structure of random size $Y_1+\cdots+Y_k$
by a method that \emph{guarantees} having \emph{exactly} $k$ components;  conditional on the event that 
 $Y_1+\cdots+Y_k=n_0$,  we have sampled  $\bD(n_0,k)$, the distribution of component counts induced from taking all $p(n_0,k)$ possible assemblies equally likely.   Note that when using \eqref{CRY T},  it is possible that a single $Y$ value will be larger than $n$;  if one were to be horrified by such an occurence, and tempted to throw away such samples, it would be preferable to use \eqref{CRY T_n} instead.

Finally, if the goal is to sample \emph{exactly}  from the distribution of $\bD(n,k)$,  one strategy is ``hard rejection/acceptance'' sampling, which in this context would repeatedly proposing a value of
$(Y_1,\ldots,Y_k)$,  testing to see if  $Y_1+\cdots+Y_k=n$,  and accepting if so, otherwise restarting.  With $p(x,k,n):= \p_x(Y_1+\cdots+Y_k=n)$, the expected number of proposals before finding an acceptable one is $1/p(x,k,n)$,  so the extended saddle heuristic, which corresponds to picking $x$ to maximize $p(x,k,n)$ so that the \emph{unconditioned} $k$-sample closely resembles the conditoned distribution, also serves as a recipe for relatively efficient simulation. However,  it is possible to do much better that taking  $1/p(x,k,n)$ independent proposals per achieved sample from the exact conditioned distribution, using probabilistic-divide-and-conquer,  see \cite{PDC,PDCDSH}.

We should mention an alternate strategy, which is not limited to assemblies.  Namely,  the two-variable generating function in \eqref{theta x} corresponds to a two-parameter family of distributions for independent $Z_1,Z_2,\ldots$,   see
\cite[Section 8]{IPARCS}.  Here we take mixed notation:  dummy variable $z$ in the generating function corresponds to $x$  when tilting, i.e., biasing with respect to $x^{T_n}$,  
(the proof of Lemma \ref{lemma small x} is an example of the use of this),
but we use $\theta$ for \emph{both} roles, first as 
the dummy variable in the generating function in \eqref{theta x}, 
\emph{and} as second the tilting parameter, when biasing with respect to $\theta^{K_n}$,   with $T_n:= Z_1+2Z_2+\cdots+nZ_n$
and $K_n := Z_1+Z_2+\cdots+Z_n$.   These tilted distributions have the property, similar to 
\eqref{CR with k T},  that  for every $(\theta,x)$,
\begin{equation}\label{CR theta x}
\text{ {\bf CR}:}  \ \ \ \cL((D_1(n,k),D_2(n,k),\ldots) )= \cL_{\theta,x}((Z_1,Z_2,\ldots) \ | \ T_n=n,  
 K_n=k).
\end{equation}   
Relation \eqref{CR theta x}  applies not just to assemblies,  but also to multisets and selections.
For \emph{selections}, as characterized by \eqref{selection p},  the $(\theta,x)$ distribution for $Z_i$ is Binomial($m_i,\theta \, x^i/(1+\theta \, x^i))$.  For \emph{multisets}, as characterized by \eqref{multiset p},  the $(\theta,x)$ distribution for $Z_i$ is Negative Binomial($m_i, \theta \, x^i)$,  so that the distribution of $Z_i$ is the $m_i$-fold convolution of the Geometric  (starting from zero, with ratio $ \theta \, x^i$) distribution on $\{0,1,2,\ldots\}$, that is, the distribution of $G$ with $\p(G \ge k)=(\theta \,x^i)^{k}, k=0,1,2,\ldots$.
For \emph{assemblies},  the $(\theta,x)$ distribution for $Z_i$ is Poisson with parameter
$\lambda_i(\theta,x) = \theta \, m_i \, x^i / i!$;  this is an extension of \eqref{lambda}.   Notice that for use in \eqref{Y},
the Poisson process construction for guaranteeing $k$ arrivals,  substituting $\lambda_i(\theta,x)$ for
$\lambda_i(x)$  yields no change, as the factor $\theta$ \emph{cancels} from numerator and denominator   --- in the second case of \eqref{Y}, the fraction naturally changes to $\lambda_i(\theta,x)/\sum_{j \ge 1} \lambda_j(\theta,x) =( \theta \, m_i \, x^i/i!) /\, (\theta \,M(x))$.

The extended saddle point heuristic says that, with parameters $\theta,x$ chosen so that  $\ext K_n$ is near $k$  and $\ext T_n$ is near $n$,  then the conditioning on the right hand side of \eqref{CR with k T} has a mild effect,  so the unconditioned $(Z_1,Z_2,\ldots,Z_n)$ is, with respect to various functionals, a good approximation to the distribution of $\bD(n,k)$.   Hence,  without asking for theorems to certify how good an approximation one gets, there is the following: \\
\noindent {\bf The $k$-Boltzmann sampler, version 2}.  To quickly generate random surrogates for the component structure $\bD(n,k)$,  find parameters $\theta,x$  so that   $\ext K_n$ is near $k$  and $\ext T_n$ is near $n$, generate the independent process $(Z_1,Z_2,\ldots,Z_n)$ under its $(\theta,x)$ law,  and accept these counts, even if $Z_1+2Z_2+\cdots+nZ_n$ is not equal to $n$, \emph{or}
$Z_1+Z_2+\cdots+Z_n$ is not equal to $k$.

\section{Application:  low rank structures}

In the context of the Tutte polynomial and its rank and nullity expansion,  the \emph{rank} of a graph with $n$ vertices and $k$ connected components is defined to be $r := n-k$.   Graphs are an instance of a decomposable combinatorial structure,  and since we are working with situations such as $k = n - \lfloor t \sqrt{n} \rfloor$,
in which it is convenient to focus on $r=n-k$   and simply write $r \sim t \sqrt{n}$,  we shall henceforth use the term \emph{rank}   for $n-k$ in the broader context of decomposable  structures of size $n$ with $k$ components.

The extreme case of a low rank structure is the case $r=0$,  which forces all components to have size 1.   Rank 
$r=1$ forces the structure to have a single component of size 2, and all other components of size 1.  Jumping up a little,  rank $r=3$ implies that the component structure be one of three possible types, namely $\lambda = (2,2,2,1,1,\ldots,1)$ or $\lambda = (3,2,1,1,\ldots,1)$ or $\lambda=(4,1,1,\ldots,1)$.   Clearly, one imagines removing 1 from each part.   So we define the \emph{copartition} of a partition $\lambda$  to be the partition 
$\overline{\lambda}$ formed from $\lambda$ by removing parts of size 1 and reducing each remaining part by 1 ---
that is, erasing the first column (or row, depending on one's choice of orientation) from the Ferrers diagram.   For future reference,
\begin{equation}\label{copartition}
\lambda = (\lambda_1,\lambda_2,\dots,\lambda_k) \text{ with } \lambda_j>1,\lambda_{j+1}=\dots=\lambda_k=1 
\end{equation}
has copartition 
\begin{equation}\label{copartition 2}
 \overline{\lambda} = (\lambda_1-1,\lambda_2-1,\dots,\lambda_j-1).
\end{equation}
Thus, the three examples we gave for rank 3  correspond to $\overline{\lambda}=(1,1,1)$ and $\overline{\lambda}=(2,1)$ and $\overline{\lambda}=(3)$.  In general,  for any $0 \le r < n$,  the number of component types for a structure of size $n$ with rank $r$ is  $p_r$, the number of integer partitions of $r$, whose asymptotics as $r \to \infty$ are given by the Hardy--Ramanujan formula \eqref{HR}.

Partitions of $n$, with rank $r$, are in one to one correspondence with partitions of $r$;   the largest component of $\lambda$ is one more that the largest component of $\overline{\lambda}$.   Hence, as $r,n \to \infty$ with $n > r$ but otherwise logically independent of the relative growth rates of $r $ and $n$, the size of the largest component of a rank $r$ partition of the integer $n$ grows like (one plus)  the largest part of a random integer partition of $r$.   Erd\H{o}s and Lehner \cite{ErdosLehner} described the growth of the largest part of a random $n$ partition:
\begin{equation}\label{lehner}
   \text{ with } c = \pi/\sqrt{6}, \   \text{ as } n \to \infty, \ L_1  \sim \frac{1}{2c} \sqrt{n} \log (n)
\end{equation}
in probability,  so sampling uniformly from the $p(n,k)=p(n,n-r)$ integer partititions of rank $r$ we have
 \begin{equation}\label{lehner r}
    \text{ as } n,r \to \infty, \ L_1  \sim \frac{1}{2c} \sqrt{r} \log (r).
\end{equation}
This behavior is in sharp contrast with the universal behavior for low rank assemblies, including set partitions, given by Theorems \ref{thm 1} -- \ref{thm 2}.

\begin{question}\label{poly question}
Consider polynomials over $\BF_q$, as in \eqref{polys}, in the low rank regime:  pick uniformly from the $p(n,k)$ monic polynomials of degree $n$ having exactly $k=n-r$ irreducible factors, where $r \sim \sqrt{n}$, or more generally  $r \sim t n^\alpha$ for $\alpha \in (0,1)$ with fixed $t \in (0,\infty)$.
Determine the behavior of $L_1$, the largest degree of an irreducible factor. 
\end{question}

For Question \ref{poly question} above,  it is easy to see that the behavior described in Theorem \ref{thm 1} does not \emph{hold}.   This is striking,  since  random polynomials over $\BF_q$, and random permutations, behave very similarly  with respect to typical  aspects of the component structure.  The key difference is that,  with respect to component structure, permutations are an \emph{assembly},  while polynomials are a \emph{multiset}.

\ignore{  
Another example to show that the behavior governed by  Theorems \ref{thm 1} -- \ref{thm 2} is essentially a property of \emph{assemblies}  is polynomials over $\BF_q$, as in \eqref{polys}.  With all $q^n$ monic irreducible polynomials of degree $n$
taken to be equally likely, the component count $\bC(n)$,  with $C_i(n)$ reporting the number of monic irreducible factors of degree $i$,
is qualitatively very similar to the cycle type $\bC(n)$ of a random permutation with all $n!$ possibilities taken to be equally likely.
Now consider the low rank regime, say $r \sim \sqrt{n}$. For random permutatations,  \eqref{expected at ell=1} simplifies to
the assertion that $D_3(n,n-r)$ converges in distribution to Poisson with mean $4/3$, the chance of having only fixed points and two-cycles converges to $e^{-4/3}$,   and Theorem \ref{thm 1} includes the
assertion that $D_4(n,n-r)$ converges in probability to zero.   But for random polynomials of rank $r$,  i.e., selecting uniformly from those monic irreducible polynomials of degree $n$ having exactly $n-r$ irreducible factors
}  

\subsection{Low rank assemblies}

Recall from Section \ref{sect Poisson}  that the component structure of  \emph{assemblies} of size $n$ with $k$ components can be handled using  independent random $\BN$-valued  variables $Y_1,\ldots,Y_k$,   conditional on $Y_1+\cdots+Y_k=n$.  Since we focus on \emph{rank} and the associated copartition,
we consider $X_j := Y_j -1$, so that as events,  
\begin{equation}\label{r target}
\{Y_1+\cdots+Y_k=n\} = \{X_1+\cdots+X_k =r \}.
\end{equation}

Assume $m_1>0$ and $M$ has a nonzero radius of convergence.  It is then immediate that as $x \to 0$, $M(x) \sim m_1 x$. 

For any positive $x$ less than the radius of convergence of $M(\cdot)$, 
the $x$-distribution of $X := Y-1$ is given by
\begin{equation}\label{dist X}
 p_i \equiv p_i(x) \equiv  \p_x(X=i) = \frac{ m_1 x}{ M(x)} \ \frac{m_{i+1}x^i}{m_1 (i+1)!}, \ i=0,1,2,\ldots .
\end{equation}  
In \eqref{dist X}, the choice to also factor out $m_1$  is so that we may write
$$
   (p_0,p_1,p_2,\ldots)  \propto (1, \frac{m_2}{2 m_1} x, \frac{m_3}{6 m_1} x^2, \ldots),
$$    
with the constant of proportionality being $\frac{ m_1 x}{ M(x)} \to 1$ as $x \to 0+$.  

\begin{lemma}\label{lemma small x}  
Assume that $m_1 >0$ and $M$ has radius of convergence $R \in (0,\infty]$.
As $x \to 0+$,
\begin{equation}\label{claim 1}
\ex(X) \sim \frac{m_2}{2 m_1} \, x, 
\end{equation}
and also 
$$
\p(X=0) \to 1,   \ \p(X = 1) \sim \ex X, \ \text{and } \p(X \ge 2) = O(x^2).
$$
\end{lemma}
\begin{proof}
Starting from \eqref{def M}, let $A(z) := M(z)/z = \sum_{i \ge 0} m_{i+1} z^i/(i+1)! = m_1 + (m_2/2)z +(m_3/6)z^2+ \cdots$.
The probability generating function for the $x$-distribution of $X$ in \eqref{dist X}, expressed with dummy variable $z$, is
$$
G(z) \equiv G_{x,X}(z)  := \ex z^X = \frac{A(xz)}{A(x)}
$$
with $G'(z)= x A'(xz)/A(x)$, hence $\ex X = G'(1)=x A'(x)/A(x)$.   Obviously, as $x \to 0$, $A(z) \to m_1$ and $A'(z) \to m_2/2$.    This establishes \eqref{claim 1}, and the remaining claims are obvious.
\end{proof}

Now take $X,X_1,\ldots,X_k$ to be i.i.d.,   and for $i \ge 0$,  let 
\begin{equation} \label{eq:N}
N_i := \sum_{j=1}^k 1(X_j=i) 
\end{equation}
 be the number of $i\,$s in the sample of size $k$.
  Note that for all outcomes, $N_0+N_1+\cdots = k$. \emph{Warning:}  for this application to low rank only,  we shift notation 
used in \eqref{CRY T}, without changing the letter $N$, \emph{from} $N_i \equiv N_i(k) := $ the count, how many of $Y_1,\ldots,Y_k$ are equal to $i$,  for $i\ge 1$, \emph{to} $N_i := $ the count, how many of $X_1,\ldots,X_k$ are equal to $i$,  for $i\ge 0$.   
Hence the conditioning event \eqref{r target} will now be expressed as  
\begin{equation}\label{target via counts}
\{ \omega:   X_1 + \cdots + X_k = r \}  = \{ \omega:  0 N_0 + N_1 + 2 N_2 + \cdots = r \}.
\end{equation}
The joint distribution of $(N_0,N_1,\ldots)$ is similar to a multinomial distribution, and indeed, for any list of $m$ disjoint sets $B_1,\ldots,B_m$ whose union is $\BZ_+$,  the \emph{lumped together} count vector $(\N_1,\ldots,\N_m)$ where  $\N_j := \sum_{i \in B_j} N_i$,  has a genuine multinomial
distribution, corresponding to $k$ tosses of an $m$-sided die,  in which face $j$ has probability
$\overline{p}_j := \sum_{i \in B_j} p_i$.   A particular case of the preceeding remark,   with $m=2$ and $B_1=$ the singleton $\{i\}$,  is that the distribution of $N_i$ is Binomial$(k, p_i)$.  

Our goal is, for an appropriate choice of $x$,  to first approximate $\p_x(N_1 + 2 N_2 + \cdots = r)$
and then approximate the conditional distribution of $(N_0,N_1,N_2,\ldots)$ given that
$N_1 + 2 N_2 + \cdots = r$ by its unconditional distribution.

\subsection{The critical regime for having components of size 3}

We have already assumed that $m_1 >0$ and $M$ has a strictly positive radius of convergence, and now we \emph{also}  assume that $m_2,m_3 > 0$.

Fix  $t \in (0,\infty)$ and fix any sequence $k(1),k(2),\ldots$ such that  $r   \equiv r(n) :=  n-k$ satisfies
\begin{equation}\label{critical 3}
  r \sim t \sqrt{n}.
\end{equation}  
Observe that this entails $r=o(n)$ and hence $k \sim n$,  so we also have $r \sim t \sqrt{k}$.
In particular, we have $r/k \to 0$ as $n \to \infty$.    We want to find $x$ so that  \eqref{dist X}
has  $k p_1 = r$. Lemma \ref{lemma small x} implies that for small $x$,   $p_1 \sim \ex X \sim m_2 x/(2 m_1)$,  so the first guess  $x_0=(2 m_1/m_2)  \, r/k$ would have $k p_1(x_0) \sim r$.
We have $x_0 \to 0$,  so for sufficiently large $n$,  $x_0 < R/2$,  and using $m_1 x/M(x) \to 1$, we can find $x$ relatively close to $x_0$, such that $k \, p_1(x)=r$.

For sufficiently large $n$, define
\begin{equation}\label{good x}
  x \equiv x(n) = \text{ the solution of } p_1 = \frac{r}{k},  \text{ so that }  \ex N_1 = k \, p_1 = r,
\end{equation}
and in case there is more than one positive solution, amend \eqref{good x} to choose the smallest positive solution.    It follows, from Lemma \ref{lemma small x} combined with \eqref{critical 3},  that this $x \equiv x(n)$ satisfies
$$
x \sim \frac{2 m_1}{m_2} \ \frac{r}{k}  \sim \frac{2 m_1}{m_2} \frac{t}{\sqrt{k}}
$$  
and hence 
\begin{equation}\label{lambda 3}
   k \,  p_2 \sim k \frac{m_3}{6 m_1} x^2 \to \frac{2 m_1 m_3}{3 m_2^2} t^2.
\end{equation}

We now sketch the overall argument for Theorem \ref{thm 1},  saving minor details for the formal proof.   
From \eqref{lambda 3}  we see that $N_2$,  whose marginal distribution is exactly Binomial$(k,p_2)$,
will be approximately Poisson  with mean given by the r.h.s.~of \eqref{lambda 3}.
We have $\e(3N_3+4N_4+\cdots)=O(x)$ times $\e N_2$,  so those contributions will be negligible.
Thus, in typical outcomes,  with \emph{small} exceptional probability, both  $K_0 := N_2 + N_3 +\cdots$  and  $R_0 := 2 N_2 + 3 N_3 + \cdots$ will be zero, or small positive integers.  Conditional on $K_0 = k_0$,  and also conditioning on which indices in the $k$-sample contributed to $K_0$,  there are $k-k_0$  rolls of the die, on which the event in \eqref{target via counts}  occurs if and only if each of those $k-k_0$ rolls shows face 0 or 1,  and the total number of 1s  is $r-r_0$.   But this is just asking for Binomial($k-k_0, p_1/(p_0+p_1))$ to have some value  $r-r_0$ close to its mean,  and the conditional probability is
close to $1/\sqrt{2 \pi k \, p_1}$,  regardless of which small values $k_0,r_0$.   So the overall 
probability of the event in \eqref{target via counts} is asymptotic to $1/\sqrt{2 \pi k \, p_1}$,  and conditional on that event, $N_2$ is still close to Poisson with mean given by the right side of 
\eqref{lambda 3}.  This shows that in the earlier consideration,  the \emph{small} probability for the exceptional event
needs to be $o(1/\sqrt{r})$, and this is sufficient to prove \eqref{largest at ell=1}.  Recall that via the shift $X = Y-1$,  the count $N_2$ here actually corresponds to blocks of size 3.

\begin{theorem}\label{thm 1}
Consider an assembly as governed by \eqref{def M} and \eqref{exponential relation}, and assume further that $m_1,m_2,m_3>0$ and $M(\cdot)$ has a strictly positive radius of convergence.
Fix $t>0$ and a sequence $k(1),k(2),\dots \ge 1$ such that $n-k(n) \sim t \sqrt{n}$.  Given $n$, pick an assembly uniformly from the $p(n,k)$ choices having exactly $k$ components.
Then, with probability tending to 1,  the largest component $L_1$ has size 2 or 3,   the number of components of size 3 has distributional limit given by 
\begin{equation}\label{expected at ell=1}
D_3(n,k) \to^d \text{ Poisson with mean }  \lambda \equiv \lambda(t,M) := 
\frac{2 m_1 m_3}{3 m_2^2} t^2,
\end{equation}
and hence for an assembly of size $n$ chosen uniformly from the $p(n,k)$ possibilities with $k$ components, we have 
\begin{equation}\label{largest at ell=1}
\p(L_1=2) \to \exp(-\lambda(t,M)).
\end{equation}
The error in the approximation \eqref{largest at ell=1},   and indeed the total variation distance
between the ingredients in \eqref{expected at ell=1}, i.e., $\dtv(D_3(n,k), \text{Poisson}(\lambda))$,
is at most  $ 
 O_t((\log^2 n)/\sqrt{n})$.

\end{theorem}

\begin{proof}
Pick $x$ as per \eqref{good x} \emph{modified} slightly, so that 
$$
   k \frac{p_1(x)}{p_0(x)+p_1(x)} = r.
$$   
The new choice of $x$ is asymptotic to the old choice, and we still have a 
result like \eqref{lambda 3}, namely  that $\e_x N_2 = O_t(1)$.  The implicit constant in the big O  depends both on $t$  and $M(\cdot)$, but we highlight the only the dependence on $t$,  since we consider the assembly $M$ as fixed.  
The binomial distribution of $N_2$ satisfies the Hoeffding bound (see \cite[equation (41)]{baxendale}), for all $y >1$,  $\p(N_2 \ge y \, \e N_2) \le 
y^{-y\, \e N_2} e^{(y-1) \e N_2}$,  so we can pick $r_2 = O_t(\log n)$ and $n_2 \equiv n_2(t)$ so that the \emph{bad} event $B_2 :=\{ 2N_2 \ge r_2 \}$ has $\p_x(B_2) < 1/n$ for all $n > n_2$.
Similarly, but not as delicate, one can pick $r_3 = O_t(\log n)$ and $n_3 \equiv n_3(t)$ so that the \emph{bad} event $B_3 :=\{ 3N_3 + 4N_4+\cdots \ge r_3 \}$ has $\p_x(B_3) < 1/n$ for all $n > n_3$.   Combining, the bad event $B := B_2 \cup B_3$ has $\p_x(B)<2/n$ for all sufficiently large $n$,  and on the complementary event, $B^c$, with $k_1 =r_1 := r_2+r_3 = O_t(\log n)$ we have
$K_0 :=N_2+N_3+\cdots \le k_1$ and $R_0:=2 N_2 + 3 N_3 +\cdots \le r_1$.

Conditional on $\{K_0=k_0\}$ and further conditioning on \emph{which} of the $k-k_0$ indices $j\in [k]$ did not contribute to $K_0$, i.e.,  those $j$ for which $X_j=0$ or 1, we have $k-k_0$ independent trials where $p := \p(X_j=1)=1-\p(X_j=0) = p_1/(p_0+p_1)$,  and from \eqref{good x} 
we have $p\sim t/\sqrt{k}$.   For the binomial($k,p$) distribution,  the target $r$ is exactly the mean,
 the binomial point probability at the mean is approximately
$1/\sqrt{2 \pi k \, p } = 1/\sqrt{2 \pi r}$,  with error controlled by Stirling's formula;  the relative error is $O(1/\sqrt{k})$.  Likewise, for the perturbations, where $k$ is replaced by $k-k_0$ and the target is replaced by $r-r_0$,   with $0 \le k_0,r_0 \le k_1 = O_t(\log k)$,  (recalling that $k \sim n$), the probability that Binomial$(k-k_0,p)$ hits the point $r-r_0$ is asymptotically $1/\sqrt{2 \pi r}$. 
For the relative error, the main contribution comes from the target being at most $r_0$ from the center,  and since the variance is order of $r$, the resulting relative error is $O(r_0^2/r) = O_t((\log n)^2/\sqrt{n}).$ 
 
The second paragraph of this proof shows that the \emph{contribution} to $\p_x(N_1+2N_2+\cdots = r)$
from the event $B^c$ is  asymptotically $1/\sqrt{2 \pi r}$, with relative error at most  
 $O_t((\log^2 n)/\sqrt{n})$.  Our bound $\p_x(B)<2/n$ from the first paragraph of this proof is of a smaller order,  so the net result is that $\p_x(N_1+2N_2+\cdots = r) \sim 1/\sqrt{2 \pi r}(1+O_t((\log^2 n)/\sqrt{n}))$.    Now that $\p_x(N_1+2N_2+\cdots = r)$ has been estimated asymptotically, the same argument from the second paragraph shows that the \emph{conditional} probability that
$N_2=m$ given that $N_1+2N_2+\cdots = r$ is relatively close to the unconditional probability, again with relative error that is $O_t((\log^2 n)/\sqrt{n})$.   Finally, the marginal distribution of
$N_2$ is exactly Binomial($k,p_1(x))$, with mean asymptotically $O_t(1)$, as given in detail by
\eqref{lambda 3} --- so the total variation distance from this binomial marginal distribution, to its Poisson approximation, is $O(1/k)=O_t(1/n)$.  This proves both \eqref{expected at ell=1} and
\eqref{largest at ell=1} and even shows that $\p(L_1=2) -\exp(-\lambda(M,t)) = 
 O_t((\log^2 n)/\sqrt{n})$.
 \end{proof}

\begin{corollary}\label{cor 1}
Consider an assembly as governed by \eqref{def M} and \eqref{exponential relation}, and assume further that $m_1,m_2,m_3>0$ and $M(\cdot)$ has a strictly positive radius of convergence.
Fix $t>0$ and a sequence $k(1),k(2),\dots \ge 1$ such that $r :=n-k(n) \sim t \sqrt{n}$. 
Then  
\begin{equation}\label{amazing}
p(n,k) \sim  \frac{n^{2r} \, m_1^{n-2r}  \, m_2^r}{r! \, 2^r} \
 \exp\left(-t^2 \left( 2 - \frac{2 m_1 m_3}{3 m_2^2} \right) \right),
\end{equation}
and the relative error in \eqref{amazing} is  $O_t((\log^2 n)/\sqrt{n})$.
\end{corollary}

\begin{proof}
With rank $r := n-k$, let $\ba=(n-2r,r,0,0,\ldots)$.  Note, this specifies the integer partition with
$(n-2r)+r = k$ parts, and is a partition of $1 \times (n-2r) + 2 \times r = n$,  so $N(n,\ba)$ is one of the contributions to $p(n,k)$, as discussed in the first paragraph of Section \ref{sect n k}.   From \eqref{assembly count}, and writing the falling power $x$ falling $i$ as  $(x)_i$, we have
\begin{equation}\label{type 1}
N(n,\ba) = \frac{n! \, m_1^{n-2r}  \, m_2^r}{(n-2r)! \, r! \, 2^r} = 
\frac{(n)_{2r} \, m_1^{n-2r}  \, m_2^r}{r! \, 2^r}.
\end{equation}
Using $r \sim t \sqrt{n}$ and the usual asymptotic for the birthday problem, that $(n)_i/n^i \sim \exp(-i^2/(2n))$ whenever 
 $i=o(n^{2/3})$, we have $(n)_{2r}/n^{2r} \to \exp(-2 t^2)$. 
Notice that, on the left side of \eqref{largest at ell=1}, we have
$$
\p(L_1=2) = \frac{N(n,\ba)}{p(n,k)}, \text{ so that }  p(n,k) =  \frac{N(n,\ba)}{\p(L_1=2)     }.
$$
Combining this with \eqref{expected at ell=1} and \eqref{largest at ell=1} yields the desired result.
\end{proof}

\begin{remark}\label{big Oh 1} 
 The upper bound on the relative error, proved in Theorem \ref{thm 1} and inherited by Corollary \ref{cor 1},
 is  $O_t((\log^2 n)/\sqrt{n})$.  
 This reflects our desire to be succinct. 
 We believe that the true error is  order 
 $ 1/\sqrt{n}$, and will make formal conjectures out of this,  with Conjectures \ref{conj O1} --- \ref{conj O3}.  Note that we are working under the regime $r \sim t \sqrt{n}$,   so  $ 1/\sqrt{n} \sim t/r$.
 
 \end{remark}
 
\begin{conjecture}\label{conj O1}
Under the hypotheses of Theorem \ref{thm 1},  the result \eqref{largest at ell=1} can be improved to 
\begin{equation}\label{largest sharper 1}
\p(L_1=2) =  \exp(-\lambda(t,M))\left( 1+O_t(1/r) \right)  \end{equation}
$$  =  \exp(-\lambda(t,M))\left( 1+O_t(1/\sqrt{n}) \right).$$
\end{conjecture}

\begin{conjecture}\label{conj O2}
Under the hypotheses of Theorem \ref{thm 1},  the \emph{true} order of error in \eqref{largest sharper 1}
is order of $1/r$, in the sense that there is a function $C:  (0,\infty) \to \BR$, depending on $m_1,m_2,m_3$, such that
when $n\to \infty$ and $r = \lfloor t \sqrt{n} \rfloor$, we have 
\begin{equation}\label{largest sharper 2}
\p(L_1=2) =
  \exp(-\lambda(t,M))  + \frac{C(t)}{r } + o_t(1/r).  \end{equation}
 \end{conjecture}

\begin{conjecture}\label{conj O3}
 The function $C(\cdot)$ for use in \eqref{largest sharper 2}
 is given explicitly by
 \begin{equation}\label{best guess}
   C(t) =  (2 \lambda^2+\lambda)  - 2 t^2  \lambda(t,M)  - t^2 \frac{m_4}{4} \, \lambda. 
\end{equation}
 \end{conjecture}

\begin{remark}\label{guess}
The expression in \eqref{best guess}, albeit  highly technical, is a plausible attempt to name all the order of $1/r$ contributions to the relative error between $\p(L_1=2)$ and $  \exp(-\lambda(t,M))$.  We view $C(t)$
as a sum with three terms. 

 For the first term, $(2 \lambda^2+\lambda)$, consider outcomes where $L_1 \le 3$;  these are
($N_2=j,N_1=r-2j,N_0=k-r-j$), for $j=0,1,2,\dots$.  In the second paragraph of the proof of Theorem \ref{thm 1}, these correspond to the situation where $k_0=j$, the conditional distribution of $N_1$ is exactly Binomial($k-j,r/k)$, and the target value is $r-2j$. The relative error between  Binomial($k-j,r/k)[r-2j]$ and Binomial($k,r/k)[r-2j]$ is $O(1/k)=o(1/r)$, negligible here.    For Binomial($k,r/k)$,  the relative difference between the mass at $r$ and the mass at $r-m$, i.e., between  Binomial($k,r/k)[r]$ and  Binomial($k,r/k)[r-m]$,  is $ 1-(m)_r/m^r  + o(1/r) = {m-1 \choose 2}/r + o(1/r)$.  We use this with  $m=2j = 2 N_2$.  Under the Poisson approximation where the distribution of $N_2$ is taken to be Poisson($\lambda$), we have $\e {2 N_2 -1 \choose 2} = 2 \lambda^2+\lambda$.

For the second term,  consider that the choice used for $x$ in the proof of Theorem \ref{thm 1},  which is described even more explicitly by \eqref{good x 24},   leads to $\lambda' := \e N_2  = n/(n-2r) \times \lambda$.  The relative error
in approximating  $\exp(-\lambda')$ by $\exp(-\lambda)$  is $2 t^2  \lambda /r + o(1/r)$.   

For the third term,  consider the event $N_3=1$,  corresponding to the assembly having exactly one component of size 4.     In Remark \ref{perspective},  the most likely representative of this event is described by $\ba'''$,  leading to the plausible belief that,  with the event in \eqref{target via counts} denoted as $G$,
 
$$
\frac{\p(N_3=1,  G)}{ \p( G)}  
\sim  \frac{N(n,\ba''') }{N(n,\ba) }  \sim  \frac{ m_4}{4} \, \frac{r}{n}  \sim  t^2 \frac{m_4}{4} \, \lambda \ \frac{1}{r}.
$$
\end{remark}

 \begin{remark}\label{perspective}  To give perspective on the meaning and extent of sharpness of the upcoming Theorem \ref{thm 1.5}, we consider four particular partition types for an assembly of size $n$ to have rank $r$.  In each case, we describe the rank $r$
partition of $n$  first by its counts $\ba$,  then via the notation $\lambda = 1^{a_1}2^{a_2}\dots$, and finally by the copartition $\overline{\lambda}$ as described by \eqref{copartition} and \eqref{copartition 2}.   The first type is familiar from the proof of Corollary \ref{cor 1}.
\[\begin{array}{llllll}
\ba&=(n-2r,r,0,\dots) & \lambda&=1^{n-2r}2^r & \overline{\lambda} &= 1^r  \\
\ba'&=(n-2r+1,r-2,1,0,\dots) & \lambda'&=1^{n-2r+1}2^{r-2}3^1 & \overline{\lambda'} &= 1^{r-2}2^1  \\
\ba''&=(n-2r+2,r-4,2,0,\dots) & \lambda''&=1^{n-2r+2}2^{r-4}3^2 & \overline{\lambda''} &= 1^{r-4}2^2 \\
\ba'''&=(n-2r+2,r-3,0,1,0,\dots) & \lambda'''&=1^{n-2r+2}2^{r-3}4^1 & \overline{\lambda'''} &= 1^{r-3}3^1  
\end{array}\]
The exact count of how many $M$-assemblies have type $\ba$, the first case in the list above, is given in \eqref{type 1}.
The corresponding exact counts for the next three cases are

\begin{equation}\label{type 2}
N(n,\ba') = 
\frac{(n)_{2r-1} \, m_1^{n-2r+1}  \, m_2^{r-2} m_3}{(r-2)! \, 2^{r-2} \, 3!}.
\end{equation}

\begin{equation}\label{type 3}
N(n,\ba'') = 
\frac{(n)_{2r-2} \, m_1^{n-2r+2}  \, m_2^{r-4}m_3^2}{(r-4)! \, 2^{r-4} 2! \, (3!)^2 }.
\end{equation}

\begin{equation}\label{type 4}
N(n,\ba''') =  
\frac{(n)_{2r-2} \, m_1^{n-2r+2}  \, m_2^{r-3}m_4}{(r-3)! \, 2^{r-3} \, 4!}.
\end{equation}

Considering the ratios of each of the above three with $N(n,\ba)$, which for $\ba=(n-2r,r,0,\dots)$ is the exact count of $M$-assemblies of rank $r$ and with $L_1=2$, for any $r \ge 1$,  we see that 
\begin{align*}
\frac{N(n,\ba') }{N(n,\ba) } & = \frac{ m_1 \, (r)_2 \, 2^2 \, m_3}{(n-2r+1) m_2^2 \, 3!}, \\ \\
\frac{N(n,\ba'') }{N(n,\ba) } & = \frac{ m_1^2 \, (r)_4  \, 2^4 \, m_3^2 \,  }{(n-2r+1)(n-2r+2) (3!)^2 \,m_2^4}, \\ \\
\frac{N(n,\ba''') }{N(n,\ba) } & = \frac{ m_1^2 \, (r)_3 \, 2^3\, m_4}{(n-2r+1)(n-2r+2) m_2^3 \, 4!}. 
\end{align*}
In particular,  if $r\to \infty$ and $r=o(n)$ then,  with the symbol  $\asymp$ used to mean that the ratio is bounded away from zero and infinity, we have 
\begin{equation}\label{perspective 2}
\frac{N(n,\ba') }{N(n,\ba) } \sim \frac{ 2 m_1 \, m_3}{3 m_2^2}  \ \frac{r^2}{n} \asymp \frac{r^2}{n}, \qquad 
\frac{N(n,\ba'') }{N(n,\ba) }   \asymp \frac{r^4}{n^2},  \qquad  \frac{N(n,\ba''') }{N(n,\ba) }   \asymp \frac{r^3}{n^2}.
\end{equation}
This will show that the error bound in Theorem  \ref{thm 1.5} is sharp.   It hints at the job of Lemma
\ref{lemma 23},  which is to compare  $N(n,\ba)$ with the combined  count of  all rank $r$ assemblies of size $n$  having only  parts of size 1,2, and 3,  
by direct combinatorial argument.  
And it gives perspective to Lemma \ref{lemma 24},  which uses the saddle point approximation to give an upper bound on all cases, like $\ba'''$, 
involving at least one part of size 4 or greater.
\end{remark}

\begin{lemma}\label{lemma 23}
Consider $M$-assemblies with $m_1,m_2 
 > 0$.   Assume $0 < r <n/2$ and let 
\begin{equation}\label{yrn}
y=
\frac{ 2 m_1 \, m_3}{3 m_2^2}  \ \frac{r^2}{n-2r}.
\end{equation}
Picking uniformly from the $p(n,n-r)$  $M$-assemblies of rank $r$, we have 
$$
\p(L_1= 3)   \le e^{y}-1.
$$
\end{lemma}
\begin{proof}
Let
$$
\baj{j} :=  (n-2r+j,r-2j,j,0,0,\dots,0),
$$
so that the rank $r$ partitions $\ba,\ba',$ and $\ba''$ in Remark \ref{perspective}
\emph{are exactly} $\baj{j}$ for $j=0,1,2$.   Since $r>0$,  we cannot have $L_1=1$, and the event $(L_1 \le 3)$
is precisely the event 
 ($\bD(n,n-r)=\baj{j}$ for some $j$ with $0 \le j \le r/2$).   As in the proof of Corollary \ref{cor 1},
 the event $(L_1=2)$ is precisely the event $(\bD(n,n-r)=\baj{0}$).

From \eqref{assembly count} we have, for $0 \le j \le r/2$,
$$
N(\baj{j},n) = n!  \frac{ m_1^{n-(2r-j)} m_2^{r-2j} m_3^j}{(n-(2r-j))! \, 2^{r-2j} \,(r-2j)! \,(3!)^j \, j!}
$$
so that 
$$
\frac{N(\baj{j},n)}{N(\baj{0},n)}
 = \frac{m_1^{j}\, (r)_{2j} \, 2^{2j}\, m_3^j }{(n-2r+j)_j \, (3!)^j \,m_2^{2j} \, j!} \le y^j/j!,
$$
hence
$$
\p(L_1 = 3) \le \sum_{1 \le j \le r/2}  y^j/j!  \ \le e^y -1. 
$$
\end{proof}

In the next lemma, our goal is to give a completely effective lower bound on the probability of the event in \eqref{r target},  in a way that gives an asymptotically useful bound for the situation with $r=o(\sqrt{n})$.
Our saddle choice for the value of the parameter $x$ for use in \eqref{dist X} is determined by the requirement
$$
   k \frac{p_1(x)}{p_0(x)+p_1(x)} = r, \text{ equivalently }  \frac{x}{2 m_1/m_2 + x} = \frac{r}{n-r},
$$  
equivalently
\begin{equation}\label{good x 24}
x = \frac{2 m_1 r  }{m_2(n-2r)   }. 
\end{equation}
Observe that in any low rank regime, that is, whenever $r=o(n)$, 
we have $x \sim (2m_1/m_2) \, r/n$. 
We take
\begin{equation}\label{rho}
   \rho := \sup_{i \ge 3}  \left( \frac{m_i}{i!} \right)^{1/i},
\end{equation}
noting that the assumption that $M$ has strictly  positive radius of convergence is equivalent to the condition  that $\rho < \infty$.   Observe that when $r=o(n)$  the left side of \eqref{needed hyp}
is $ 2 k  \rho^3x^2/m_1  \sim (2 \rho^3/m_1)  \ n x^2 \sim (2 \rho^3/m_1)   \ r^2/n$,   so the condition that  $r=o(\sqrt{n})$ is sufficient to guarantee that  \eqref{needed hyp} holds \emph{eventually}.
Finally, observe that for the situation of interest, which is  $r=o(\sqrt{n})$,  the r.h.s.~of \eqref{24 bound} is order of $k x^3 \sqrt{r} \sim n (r/n)^3 \sqrt{r} = r^{-1/2} (r^2/n)^2$, so compared with the error contribution from Lemma \ref{lemma 23},  the error contribution from Lemma \ref{lemma 24} is of smaller order.

\begin{lemma}\label{lemma 24}
Consider an assembly as governed by \eqref{def M} and \eqref{exponential relation}, and assume further that $m_1,m_2>0$ and $M(\cdot)$ has a strictly positive radius of convergence.
Given $n,r=n-k$  let the parameter in \eqref{dist X} be given by \eqref{good x 24},
and let $\rho$ be given by \eqref{rho}.
Assume that $n,r$ satisfy
\begin{equation}\label{24 hyp}
   x \rho \le 1/2.
\end{equation}  
and
\begin{equation}\label{needed hyp}
\frac{2 k \, \rho^3x^2}{m_1 }  \le 1/2.
\end{equation}
Then, with $c_0 :=  e/\sqrt{2 \pi}$,
$$
\p(X_1+\dots+X_k = r) \ge  \frac{1}{2 c_0 \sqrt{2 \pi r}}
$$
and
\begin{equation}\label{3 bound}
  \p(  N_3+N_4+\dots>0) \le k \, \p(X \ge 3)  \le k  \frac{2 \rho^4x^3}{m_1 }
\end{equation}
and hence 
\begin{eqnarray}\label{24 bound}
 \p(L_1 \ge 4) 
   \le   k  \frac{2 \rho^4x^3}{m_1 } \
 2 c_0 \sqrt{2 \pi r}  =:  u_4(n,r).
\end{eqnarray}
\end{lemma}

\begin{proof}
Using \eqref{rho} and the assumption that $x \rho \le 1/2$,
we have
$$
\sum_{i \ge 3}  \frac{m_i x^i}{i!} \le \sum_{i \ge 3}  (\rho x)^i = \frac{(\rho x)^3 }{1-\rho x} \le 2 \rho^3 x^3.
$$
Hence in \eqref{dist X}, with this choice of $x$,  we have
$$
p_0(x)+p_1(x) = \frac{m_1 x + m_2 \, x^2/2}{M(x)}> \frac{m_1x}{m_1x + 2 \rho^3 x^3} 
=\frac{m_1}{m_1 + 2 \rho^3 x^2} 
$$
so that 
\begin{equation}\label{2 bound}
\p(X \ge 2)  = 1-(p_0+p_1) \le \frac{2 \rho^3x^2}{m_1 + 2 \rho^3 x^2} \le \frac{2 \rho^3x^2}{m_1 }. 
\end{equation}
This yields
\begin{equation}\label{bound 1}
(p_0+p_1)^k = (1-\p(X \ge 2))^k  \ge 1-k \, \p(X \ge 2)  \ge 1-\frac{2 k \, \rho^3x^2}{m_1 }.
\end{equation}
With counts $N_i$,  for $i \ge 0$, as specified just before \eqref{target via counts}, and recalling that $k=n-r$, and writing $p := p_1/(p_0+p_1)$ so that by \eqref{good x 24} we have  also $p=r/k$, we have
$$
\p(N_0=n-2r,N_1=r) = \frac{k!}{(k-r)! r!} \ p_0^{k-r}p_1^r = 
\frac{k!}{(k-r)! r!} \ (1-p)^{k-r}p^r \ (p_0+p_1)^k.
$$
The first factor on the r.h.s.~above is a point probability for a binomial distribution with mean $r$,
asymptotically $1/\sqrt{2 \pi r (1-p)}$,  and always at least $1 /(c_0 \sqrt{2 \pi r})$,  where 
$c_0 := e^1/\sqrt{2 \pi}  = 1.0844375514192 \dots$. 
The second factor on the right side above is bounded via 
\eqref{bound 1} 
and the hypothesis \eqref{needed hyp},
hence
$$
\p(X_1+\dots+X_k=n) \ge \p(N_0=n-2r,N_1=r)   \ge \frac{1}{2 c_0 \sqrt{2 \pi r}}.
$$

Using \eqref{rho} and the assumption that $x \rho \le 1/2$,
we have
$$
\sum_{i \ge 4}  \frac{m_i x^i}{i!} \le \sum_{i \ge 4}  (\rho x)^i = \frac{(\rho x)^4 }{1-\rho x} \le 2 \rho^4 x^4,
$$
and $M(x) \ge m_1 x$,  hence 
\begin{eqnarray*}
\p(N_3+N_4+\dots>0) & \le & k \, \p(X \ge 3) \\
   & = &  \frac{k}{M(x)}   \sum_{i \ge 4}  \frac{m_i x^i}{i!} \le \frac{ 2 k \rho^4 x^4}{m_1 x}.
\end{eqnarray*}
To prove \eqref{24 bound} we \emph{overpower} the requirement for the occurrence of the conditioning event:
\begin{eqnarray*}
 \p(L_1 \ge 4) &=& \p(N_3+N_4+\dots>0 | X_1+\dots+X_k = r) \\
 & = & \frac{\p(N_3+N_4+\dots>0 \emph{ and } X_1+\dots+X_k = r)}{\p( X_1+\dots+X_k = r)} \\
  & \le & \frac{\p(N_3+N_4+\dots>0 )}{\p( X_1+\dots+X_k = r)} \\
  & \le  & k  \frac{2 \rho^4x^3}{m_1 } \
 2 c_0 \sqrt{2 \pi r}.
\end{eqnarray*}

\end{proof}

\begin{theorem}\label{thm 1.5}
Consider an assembly as governed by \eqref{def M} and \eqref{exponential relation}. Given $n,r$, with $k:=n-r$ pick an assembly uniformly from the $p(n,k)$ choices having exactly $k$ components.

Assume the hypotheses of Lemmas \ref{lemma 23} and \ref{lemma 24}, i.e.,  $m_1,m_2>0$, $M(\cdot)$ has a strictly positive radius of convergence,  $0<r < n/2$, and that $x$ and $\rho$ as given by \eqref{good x 24} and \eqref{rho} satisfy \eqref{24 hyp}
and \eqref{needed hyp}.  

Then,  with $y \equiv y(n,r)$ given by \eqref{yrn}  and $u_4$ given by \eqref{24 bound},
$$
\p(L_1 \ge 3) \le  (e^y-1) + u_4(n,r)  =: z.
 $$
 Hence, in case $z<1$,
\begin{equation}\label{effective p(n,k)}
\frac{n^{2r} \, m_1^{n-2r}  \, m_2^r}{r! \, 2^r}    \le p(n,k) \le \frac{n^{2r} \, m_1^{n-2r}  \, m_2^r}{r! \, 2^r}  
\biggl(1+   \frac{z}{1-z}  
  \biggr).
 \end{equation}  
Note that when  $r=o(\sqrt{n})$, the upper bound on the relative error is  
\[ z/(1-z) \sim z \sim y \sim 2 m_1 m_3/(3 m_2^2)  \,  r^2/n \asymp r^2/n.\]
 \end{theorem}
\begin{proof}
The first statement is an immediate combination of the conclusions of Lemmas \ref{lemma 23} and \ref{lemma 24};  the second statement follows by reasoning akin to that used in the proof of Corollary \ref{cor 1}.  The asymptotic analysis was given in the paragraph preceeding Lemma \ref{lemma 24},   and the calculation in \eqref{perspective 2} shows that the upper bound is asymptotically sharp.
 \end{proof}

\begin{theorem}\label{thm 2}
Consider an assembly as governed by \eqref{def M} and \eqref{exponential relation}, and assume further that $m_1,m_2,\ldots>0$ and $M(\cdot)$ has a strictly positive radius of convergence.
Fix  a sequence $k(1),k(2),\dots \ge 1$ with $1 \le k(n) \le n$.  Given $n$, pick an assembly uniformly from the $p(n,k)$ choices having exactly $k$ components.   Write $r=n-k$.
Assume that for some $\varepsilon > 0$, $r=o(n^{1-\varepsilon})$. 

 Then 
for $\ell=1,2,\ldots$,  
\begin{itemize}

\item If $r=o(n^{\ell/(1+\ell)})$ then $\p(L_1 \le \ell+1) \to 1$.

\item If $n^{\ell/(1+\ell)} =o(r)$ then $\p(L_1 > \ell+1) \to 1$.
  
\item  If  $\liminf \log_n r > \ell/(\ell+1)$, then with $x$ given by \eqref{good x},  for each $i$ with
$1 \le i \le \ell+2$,   
\begin{equation}\label{bullet 3}
1 = \p\left(D_i(n,k) \sim    k  \ \frac{m_i x^{i-1} }{m_1 \, i!}  \right).
\end{equation}
  
\item If for fixed $t>0$ we have $r \sim  t \, n^{\ell/(1+\ell)} $  then,  with 
\begin{equation}\label{lambda ell}
\lambda \equiv \lambda(t,\ell,M)  :=  \frac{2^{\ell+1} m_1^\ell  \, m_{\ell+2}}{ (\ell+2)! \, m_2^{\ell+1}} \ t^{\ell+1} ,
\end{equation}
\begin{equation}\label{ell border}
\p(L_1 = \ell+1) \to e^{-\lambda} \text{  and }   \p(L_1 = \ell+2) \to 1-e^{-\lambda}.
\end{equation}
\end{itemize}
 
\end{theorem}

\begin{proof}

We will sketch two computations that differ from the situation of Theorems \ref{thm 1} and \ref{thm 1.5}.
The remaining details  for all claims are similar to  arguments given in the proof of Theorem  \ref{thm 1}, although not as delicate,  and we 
shall omit the details. 

The key computation for the borderline behavior in \eqref{ell border} is that \eqref{good x} entails
$$
x \sim \frac{2 m_1}{m_2} \ \frac{r}{k}  \sim \frac{2 m_1}{m_2} \frac{t}{k^{1/(1+\ell)}}
$$  
hence $\e N_{\ell+1}$  is asympototic to 
$$
k \, p_{\ell+1} \sim k \ \frac{m_{\ell+2}  }{m_1 (\ell+2)!}  \ x^{\ell+1} \sim
 \frac{2^{\ell+1} m_1^\ell  m_{\ell+2}}{ (\ell+2)! \, m_2^{\ell+1}} \ t^{\ell+1} =: \lambda(t,\ell,M). 
$$
Recall that $ N_{\ell+1}$  counts how many of the $X_1,\ldots,X_k$ are equal to $\ell+1$, which is the same as the number of $Y_1,\ldots,Y_k$ which are equal to $\ell+2$.

For \eqref{bullet 3},  consider a subsequence along which  $\alpha = \lim \log_n r $ \mbox{$\in (\ell/(\ell+1),1).$}   By \eqref{good x} and Lemma \ref{lemma small x},  $x \sim \frac{2 m_1}{m_2} \ \frac{r}{k} \approx n^{\alpha-1}$ (with the usual large deviation theory notation,   $a_n \approx b_n$  to mean that $\log a_n \sim \log b_n$) and hence
$$
\e N_{i-1} = k \frac{ m_1 x}{M(x)} \ \frac{m_i x^{i-1} }{m_1 \, i!} \sim  k  \ \frac{m_i x^{i-1} }{m_1 \, i!}
$$
so that  $ \e N_{i-1} \approx n^{1 + (\alpha-1)(i-1)} = n^\delta$,  with $\delta = 1 + (i-1)(\alpha-1) >0$
using $i-1 \le \ell+1$ and $1-\alpha > 1/(\ell+1)$.    The marginal distribution of $N_{i-1}$ is Binomial($k, p_{i-1}(x))$,  and a moderate deviation bound, that the probability of being more than $c \log n$ standard deviations away from the mean is $o(1/n)$, in conjunction with an overall argument that $\p(X_1+\dots+X_k=r) \asymp 1/\sqrt{r}$, establishes \eqref{bullet 3}.  (One could prove a stronger version of \eqref{bullet 3},  allowing for example
$r \sim n^{\ell/(\ell+1)} \log \log \log n$,  but then  for  $i=\ell+2$,  $\e N_{i-1}$ would grow very slowly,  and instead of brute force ``overpowering the conditioning'',  one would have to argue some approximate independence between $N_{i-1}$ and the event $X_1+\dots+X_k=r$,  as we did in the proof of Theorem  
\ref{thm 1}.)
\end{proof}

\subsection{A completely effective version of Theorem \ref{thm 1}}

Theorem \ref{thm 1} establishes
the asymptotic behavior in the critical regime for having components of size 3. In this section, we wish to highlight that by equally elementary but slightly more tedious calculations, one can just as easily provide quantitative bounds, i.e., completely effective inequalities for the relevant probabilities for all finite values of the parameters, which yield the asymptotic behavior as a corollary. 
We utilize the following lemmas in the proof of Theorem~\ref{SD thm 1}.  

\begin{lemma}{\cite[Section VI.10, Problem 34]{Feller}}\label{Poisson:lemma}
Suppose $0<p<1$ and $n\in \mathbb{N}$.  Let $\lambda = n\, p$, and define $b(k;n,p) := \binom{n}{k} p^k (1-p)^{n-k}$, and $p(k;\lambda) := \frac{\lambda^k}{k!}e^{-\lambda}$. 
Then we have
\[ p(k;\lambda)\, e^{-\frac{k^2}{n-k} - \frac{\lambda^2}{n-\lambda}} < b(k; n, p) < p(k; \lambda)\, e^{k\,\lambda/n}. \]
\end{lemma}

\begin{lemma}\label{binomial:error}
Suppose $0<p<1$, $p\,n\geq 1$ and $0< k = p\, n + h<n$.  Put
\[ \beta = \frac{1}{12k} +\frac{1}{12(n-k)}, \]
Let $b(k; n,p)$ denote the point probability that a Binomial random variable with parameters $n$ and $p$ is equal to $k$, and $q = 1-p$.
We have
\[ \sqrt{2\pi\,p\,q\,n}\, b(k; n, p) < \exp\left(-\frac{h}{2pn}+\frac{h}{2qn}-\frac{h^2}{pn}-\frac{h^2}{qn}\right). \]
We also have
\[ \sqrt{2\pi\,p\,q\,n}\, b(k; n, p) > 
\begin{cases} 
\exp\left(-\beta + h \frac{qn}{qn-h} + \frac{h^2}{qn-h} + \frac{h}{2(qn-h)}\right) & h>0, \\
\exp\left(-\beta + h \frac{pn}{pn-h} + \frac{h^2}{pn-h} + \frac{h}{2(pn-h)}\right) & h<0.
\end{cases}
\]
\end{lemma}
\begin{proof}
This lemma is an adaptation of the arguments in~\cite[Chapter 1]{Bollobas}.  Here we utilize the inequalities, valid for all $0<t<1$,
\[\begin{array}{ccccc} 
\frac{-t}{1-t} &<& \log(1-t) &<& -t \\
0 &<& \log(1+t) &<& t.
\end{array}\]
The main difference is that we do not place any added restrictions on $h$ other than $-n\,p < h < n\,p$. 
\end{proof}



\begin{lemma}\label{binomial:deviations}
Suppose $N$ is a Binomial distribution with parameters $n$ and $p$, with $\mu := \sup_n \e N < \infty$.  
Then we have 
\[ \p(N \ge \log(n)) < \frac{1}{n} \]
for all $n \geq n_0$, where we may take 
\[ n_0 = \exp\left(\mu\, e^{2} \right). \]
\end{lemma}
\begin{proof}
The Hoeffding bound for the binomial distribution, see for example~\cite[equation (41)]{baxendale}, implies that for any $y>1$, we have
\[ \p(N \ge y \, \e N) \le y^{-y\, \e N} e^{(y-1) \e N}. \]
Furthermore, taking $y = \frac{1}{\e N} \log(n)$, we then solve for $n$ in 
\[ y^{-y\, \e N} e^{(y-1) \e N} < 1/n; \]
rearranging, and taking the logarithm, we wish to satisfy
\[ \log(\e N) + 2 - \frac{\e N}{\log(n)} < \log \log n. \]
Next, we may ignore the term $\frac{\e N}{\log(n)}$ as long as $n > e^{\e N}$, in which case we obtain after exponentiating twice 
\[ n > e^{e^2 \e N}, \]
and replacing $\e N$ with $\mu$ we obtain the conclusion.
\end{proof}

\ignore{We take
\begin{equation}
   \rho := \sup_{i \ge 3}  \left( \frac{m_i}{i!} \right)^{1/i},
\end{equation}
noting that the assumption that $M$ has strictly  positive radius of convergence is equivalent to the condition  that $\rho < \infty$. 
}

\begin{lemma}\label{sum:deviations}
Let $N_3, N_4, \ldots$ be defined as in~\eqref{eq:N}, with $p_i$ given by~\eqref{dist X}, $\rho$ given by~\eqref{rho}, and $x = \frac{2m_1}{m_2}\, \frac{n-k}{k}$.  
Then
\[ \p(3N_3 + 4N_4 + \ldots \ge \log(n)) \leq \frac{1}{n} \]
for all $n \geq n_3$, where $n_3$ is the smallest value which satisfies $x\, \rho \leq \frac{1}{2}$ and 
\begin{equation}\label{n3}
 \frac{n(n-k)^3}{k^2} \leq \frac{m_1}{2\rho^4} \left(\frac{2m_1}{m_2}\right)^3.
 \end{equation}
\end{lemma}
\begin{proof}
We have 
\[ \p(3N_3 + 4N_4 + \ldots \ge \log(n)) \leq \p(N_3 + N_4 + \ldots > 0) \leq \e(N_3 + N_4 + \ldots) \leq k  \frac{2 \rho^4x^3}{m_1 }. \]
By plugging in the appropriate values for $x$ and rearranging, we obtain the result. 
\end{proof}

\begin{theorem}\label{SD thm 1}
Consider an assembly as governed by \eqref{def M} and \eqref{exponential relation}, and assume further that $m_1,m_2,m_3>0$ and $M(\cdot)$ has a strictly positive radius of convergence. 
Given $n \geq k \geq 1$, pick an assembly uniformly from the $p(n,k)$ choices having exactly $k$ components. 
Let
\[ x = \frac{2m_1}{m_2}\, \frac{n-k}{k}, \qquad b_0 = m_1\, x, \qquad b_1 = \frac{m_2 x^2}{2}, \qquad b_2 = \frac{m_3 x^3}{6}, \]
and define functions
\[ f(h,n,p) := -\frac{h}{2pn}+\frac{h}{2qn}-\frac{h^2}{pn}-\frac{h^2}{qn}, \]
\[ g(h,n,p) := h \frac{pn}{pn-h} + \frac{h^2}{pn-h} + \frac{h}{2(pn-h)}. \]
Finally, define 
\[ \lambda \equiv \lambda(n,k) := \frac{m_3}{6m_1} x^2, \qquad \mbox{ and } \qquad p := \frac{b_1}{b_0+b_1}. \] 
Then the number of components of size~$3$, $D_3(n,k)$, satisfies
\[ \p(D_3(n,k) = m) \geq \frac{\lambda^k}{k!} e^{-\lambda}\, e^{-\frac{k^2}{n-k} - \frac{\lambda^2}{n-\lambda}} \frac{n-1}{n} \frac{\exp(-\beta + g(2m+\log(n),k,p))}{ \exp(f(2m+\log(n), k, p))}, \]
and
\[ \p(D_3(n,k) = m)  \leq \frac{\lambda^k}{k!} e^{-\lambda}\, e^{\frac{\lambda m}{k}}\frac{\frac{1}{n}+\frac{1}{\sqrt{2\pi\, r}} \exp(f(\log(n), k-\log(n), p))}{\frac{n-1}{n} \frac{1}{\sqrt{2\pi\,r}} \exp(-\beta + g(\log(n), k-\log(n), p))}; \]
and the largest component $L_1$ satisfies 
\[ e^\lambda \p(L_1 = 2)  \leq  \frac{ \exp(f(0, k, p))}{\frac{n-1}{n}  \exp(-\beta + g(\log(n), k, p))} \]
and
\[  e^\lambda \p(L_1 = 2)  \geq e^{ - \frac{\lambda^2}{n-\lambda}}\, \frac{ \exp(-\beta + g(0, k, p))}{\frac{n-1}{n} \exp(f(\log(n), k, p))}\left(1-\frac{1}{n(1-b_2/(m_1 x))}\right) \]
 for all $n \geq \max(e^{\mu\, e^2}, n_3)$, where $\mu = \sup_{n,k} \frac{k\, b_1}{b_0+b_1},$ and $n_3$ is the smallest positive value such that $x \, \rho \leq \frac{1}{2}$ and satisfies~\eqref{n3}. 
\end{theorem}

\begin{proof}

We start by defining the events $B_2 :=\{ 2N_2 \ge \frac{1}{2}\log(n) \}$ and $B_3 :=\{ 3N_3 + 4N_4+\cdots \ge \frac{1}{2}\log(n) \}$.
By Lemma~\ref{binomial:deviations}, we have $\p(B_2) < \frac{1}{n}$ for all $n > e^{\mu_2\, e^2}$, where $\mu_2 := \max \frac{k\,m_3 x^2}{3!}$. 
By Lemma~\ref{sum:deviations}, we have $\p(B_3) < \frac{1}{n}$ for all $n > n_3$, where we may take $n_3$ to be the smallest positive value such that
\[ \frac{n(n-k)^3}{k^2} \leq \frac{m_1}{2\rho^4}\left(\frac{m_2}{2m_1}\right)^3. \] 
Combining, the bad event $B := B_2 \cup B_3$ has $\p_x(B)<2/n$ for all $n \geq \max(n_2, n_3)$, and on the complementary event, $B^c$, with $k_1 =r_1 := r_2+r_3 = \log n,$ we have
$K_0 :=N_2+N_3+\cdots \le k_1$ and $R_0:=2 N_2 + 3 N_3 +\cdots \le r_1$.

Conditional on $\{K_0=k_0\}$ and further conditioning on \emph{which} of the $k-k_0$ indices $j\in [k]$ did not contribute to $K_0$, i.e.,  those $j$ for which $X_j=0$ or 1, we have $k-k_0$ independent trials where $p := \p(X_j=1)=1-\p(X_j=0) = p_1/(p_0+p_1) = \frac{b_1}{b_0+b_1},$ which defines the binomial distribution for $N_1$ conditioned on the values of $N_2, N_3, \ldots$. 
We thus obtain bounds on $\p(N_1 + 2N_2 + \ldots = r)$ by splitting it up by $B$ and $B^c$.  Conditional on $B$ or $B^c$, $N_1$ is binomial with parameters $k-k_0$ and $p = p_1/(p_0+p_1)$. 
Let $R(r) := \{N_1 + 2N_2 + \ldots = r\}$, and $\beta = \frac{1}{12r} + \frac{1}{12(k-r)}$. 
Let $f(h,n,p) := -\frac{h}{2pn}+\frac{h}{2qn}-\frac{h^2}{pn}-\frac{h^2}{qn}$, and 
$g(h,n,p) := h \frac{pn}{pn-h} + \frac{h^2}{pn-h} + \frac{h}{2(pn-h)}$. 
With $k_0 = \log(n)$, we have 
\[ \p(R(r) \cap B^c) \leq \p(N_1 = r) \leq \frac{1}{\sqrt{2\pi\, r}} \exp(f(\log(n), k-k_0, p)),\]
 and 
\begin{align*}
\p(R(r) \cap B^c) & \geq  \frac{n-1}{n}\, \max_{0\leq \ell \leq \log(n)}\p(N_1 = r-\ell) \\
& \geq  \frac{n-1}{n}\, \p(N_1 = r-r_2-r_3) \\
& \geq \frac{n-1}{n} \frac{1}{\sqrt{2\pi\,r}} \exp(-\beta + g(\log(n), k-k_0, p)).
\end{align*}
Hence, 
\[\frac{n-1}{n} \frac{1}{\sqrt{2\pi\,r}} \exp(-\beta + g(\log(n), k-\log(n), p)) \leq \p(R(r))  \]
and
\[ \p(R(r)) \leq \frac{1}{n} + \frac{1}{\sqrt{2\pi\, r}} \exp(f(\log(n), k-\log(n), p)). \]

Next, we obtain bounds on $\p(N_2 = m | N_1 + 2N_2 + \ldots = r)$.  We have
\begin{align*}\displaystyle 
 \frac{\p(N_2 = m | R(r) )}{\p(N_2 = m)} & \geq  \frac{\p(B_3){\displaystyle \min_{0\leq\ell\leq \log(n)} \p(N_1 = r-2m-\ell)}}{\p(N_1 + 2N_2 + \ldots = r)}  \\
 &\qquad + \frac{\p(B_3^c) \min_{0\leq \ell \leq \log(n)} \p(N_1 = r-2m-\ell)}{\p(N_1 + 2N_2 + \ldots = r)}  \\
 & > \frac{n-1}{n} \, \frac{\p(N_1 = r-2m-\log(n))}{\p(N_1 + 2N_2 + \ldots = r)}  \\
 & > \frac{n-1}{n} \frac{\frac{1}{\sqrt{2\pi r}} \exp(-\beta + g(2m+\log(n),k,p))}{\frac{1}{\sqrt{2\pi r}} \exp(f(2m+\log(n), k, p))}.
\end{align*}
In the other direction, we have
\begin{align*}
 \frac{\p(N_2 = m | R(r))}{\p(N_2 = m)}  & < \frac{\frac{1}{n} + \p(N_1 = r-2m)}{\p(N_1 + 2N_2 + \ldots = r)} \\
 & \leq \frac{\frac{1}{n}+\frac{1}{\sqrt{2\pi\, r}} \exp(f(\log(n), k-\log(n), p))}{\frac{n-1}{n} \frac{1}{\sqrt{2\pi\,r}} \exp(-\beta + g(\log(n), k-\log(n), p))}. 
 \end{align*}
At this point we apply Lemma~\ref{Poisson:lemma} to $\p(N_2=m)$ to obtain the result. 

Let us now consider $\p(L_1 = 2) = \p(N_1 > 0, N_2 = N_3 = \cdots = 0 | R(r) ).$ 
Let $T_3 = \{N_3 = N_4 = \cdots = 0\}$. 
We have
\begin{align*}
  \p(L_1 = 2) & = \frac{ \p(N_1 >r, N_2=0, T_3)}{\p(R(r))} \\
  & = \frac{\p(N_1 = r | R(r), N_2, T_3)\ \left(1-p_2\right)^k \p(T_3 | R(r), N_2 = 0)}{\p(R(r))} \\
  & = \frac{\p(\mbox{Bin}(k,p) = r)}{\p(R(r))} \left(1-p_2\right)^k \left(1-\frac{p_3 + p_4 + \ldots}{1-p_2}\right)^k.
  \end{align*}
Whence, 
\begin{align*}
  \p(L_1 = 2) & \leq  \p(N_2 = 0) \, \frac{\frac{1}{\sqrt{2\pi\, r}} \exp(f(0, k, p))}{\frac{n-1}{n} \frac{1}{\sqrt{2\pi\,r}} \exp(-\beta + g(\log(n), k, p))} \\
 & \leq e^{-\lambda} \, \frac{ \exp(f(0, k, p))}{\frac{n-1}{n}  \exp(-\beta + g(\log(n), k, p))}.
 \end{align*}
In a similar fashion, we have  
\begin{align*}
 \p(L_1 = 2) & \geq \p(N_2 = 0) \, \frac{\frac{1}{\sqrt{2\pi\, r}} \exp(-\beta + g(0, k, p))}{\frac{n-1}{n} \frac{1}{\sqrt{2\pi\,r}} \exp(f(\log(n), k, p))}\left(1-\frac{p_3 + p_4 + \ldots}{1-p_2}\right)^k \\
 & \geq \p(N_2 = 0) \, \frac{\exp(-\beta + g(0, k, p))}{\frac{n-1}{n} \exp(f(\log(n), k, p))}\left(1-\frac{k\,(p_3 + p_4 + \ldots)}{1-p_2}\right) \\
 & \geq e^{-\lambda} \,e^{- \frac{\lambda^2}{k-\lambda}}\, \frac{ \exp(-\beta + g(0, k, p))}{\frac{n-1}{n} \exp(f(\log(n), k, p))}\left(1-\frac{k\,(p_3 + p_4 + \ldots)}{1-p_2}\right). \\
 & \geq e^{-\lambda} \,e^{- \frac{\lambda^2}{k-\lambda}}\, \frac{ \exp(-\beta + g(0, k, p))}{\frac{n-1}{n} \exp(f(\log(n), k, p))}\left(1-\frac{2k\rho^4x^3}{m_1(1-p_2)}\right).
 \end{align*}
Now whenever $n \geq n_3$, we have $\frac{2k\rho^4x^3}{m_1(1-p_2)} \leq \frac{1}{n}$, and so we have 
\[ \p(L_1 = 2)  \geq e^{-\lambda} \,e^{- \frac{\lambda^2}{k-\lambda}}\, \frac{ \exp(-\beta + g(0, k, p))}{\frac{n-1}{n} \exp(f(\log(n), k, p))}\left(1-\frac{1}{n(1-p_2)}\right).
 \]
Finally, since $p_2 = b_2 / M(x) \leq b_2 / (m_1 x)$, we obtain the final expression.
\end{proof}

\section{Acknowledgements}

The authors are grateful to Fred Kochman for helpful suggestions.

\bibliographystyle{acm}
\bibliography{piparcs}

\begin{thebibliography}{10}

\bibitem{AlmkvistUnimodal}
{\sc Almkvist, G.}
\newblock Partitions into odd, unequal parts.
\newblock {\em Journal of Pure and Applied Algebra 38}, 2-3 (1985), 121--126.

\bibitem{ABT93}
{\sc Arratia, R., Barbour, A.~D., and Tavar{\'e}, S.}
\newblock On random polynomials over finite fields.
\newblock {\em Mathematical Proceedings of the Cambridge Philosophical Society
  114}, 2 (1993), 347--368.

\bibitem{ABTbook}
{\sc Arratia, R., Barbour, A.~D., and Tavar{\'e}, S.}
\newblock {\em Logarithmic combinatorial structures: a probabilistic approach}.
\newblock EMS Monographs in Mathematics. European Mathematical Society (EMS),
  Z\"urich, 2003.

\bibitem{baxendale}
{\sc Arratia, R., and Baxendale, P.}
\newblock Bounded size bias coupling: a {G}amma function bound, and universal
  {D}ickman-function behavior.
\newblock {\em Probability Theory and Related Fields 162}, 3-4 (2015),
  411--429.

\bibitem{Stirling2}
{\sc Arratia, R., and DeSalvo, S.}
\newblock Approximation to the component sizes of low-rank set partitions and
  permutations, via rooks.
\newblock {\em Preprint\/}.

\bibitem{Stirling}
{\sc Arratia, R., and DeSalvo, S.}
\newblock Completely effective error bounds for {S}tirling numbers of the first
  and second kind via {P}oisson approximation.
\newblock {\em Annals of Combinatorics\/} (2016).

\bibitem{PDC}
{\sc Arratia, R., and DeSalvo, S.}
\newblock Probabilistic divide-and-conquer: A new exact simulation method, with
  integer partitions as an example.
\newblock {\em Combinatorics, Probability and Computing 25\/} (5 2016),
  324--351.

\bibitem{AT92}
{\sc Arratia, R., and Tavar{\'e}, S.}
\newblock The cycle structure of random permutations.
\newblock {\em The Annals of Probability 20}, 3 (1992), 1567--1591.

\bibitem{IPARCS}
{\sc Arratia, R., and Tavar{{\'e}}, S.}
\newblock Independent process approximations for random combinatorial
  structures.
\newblock {\em Advances in Mathematics 104}, 1 (1994), 90--154.
\newblock Available at http://arxiv.org/pdf/1308.3279.pdf.

\bibitem{BG}
{\sc Bender, E., and Goldman, J.}
\newblock Enumerative uses of generating functions.
\newblock {\em Indiana University Mathematics Journal 20\/} (1971), 753--765.

\bibitem{Bender}
{\sc Bender, E.~A.}
\newblock Central and local limit theorems applied to asymptotic enumeration.
\newblock {\em Journal of Combinatorial Theory. Series A 15\/} (1973), 91--111.

\bibitem{LCLT2}
{\sc Bender, E.~A., and Richmond, L.~B.}
\newblock Central and local limit theorems applied to asymptotic enumeration.
  {II}. {M}ultivariate generating functions.
\newblock {\em Journal of Combinatorial Theory. Series A 34}, 3 (1983),
  255--265.

\bibitem{LCLT3}
{\sc Bender, E.~A., Richmond, L.~B., and Williamson, S.~G.}
\newblock Central and local limit theorems applied to asymptotic enumeration.
  {III}. {M}atrix recursions.
\newblock {\em Journal of Combinatorial Theory. Series A 35}, 3 (1983),
  263--278.

\bibitem{Berlekamp}
{\sc Berlekamp, E.~R.}
\newblock {\em Algebraic coding theory}.
\newblock McGraw-Hill Book Co., New York-Toronto, Ont.-London, 1968.

\bibitem{Bollobas}
{\sc Bollob{{\'a}}s, B.}
\newblock {\em Random graphs}, second~ed., vol.~73 of {\em Cambridge Studies in
  Advanced Mathematics}.
\newblock Cambridge University Press, Cambridge, 2001.

\bibitem{BrentiUpdate}
{\sc Brenti, F.}
\newblock Log-concave and unimodal sequences in algebra, combinatorics, and
  geometry: an update.
\newblock In {\em Jerusalem combinatorics '93}, vol.~178 of {\em Contemp.
  Math.} Amer. Math. Soc., Providence, RI, 1994, pp.~71--89.

\bibitem{CanfieldElementary}
{\sc Canfield, E.~R.}
\newblock From recursions to asymptotics: on {S}zekeres' formula for the number
  of partitions.
\newblock {\em Electronic Journal of Combinatorics 4}, 2 (1997), Research Paper
  6, approx. 16 pp. (electronic).
\newblock The Wilf Festschrift (Philadelphia, PA, 1996).

\bibitem{ChelluriRichmond}
{\sc Chelluri, R., Richmond, L.~B., and Temme, N.~M.}
\newblock Asymptotic estimates for generalized {S}tirling numbers.
\newblock {\em Analysis. International Mathematical Journal of Analysis and its
  Applications 20}, 1 (2000), 1--13.

\bibitem{Daniels}
{\sc Daniels, H.~E.}
\newblock Saddlepoint approximations in statistics.
\newblock {\em Annals of Mathematical Statistics 25\/} (1954), 631--650.

\bibitem{PDCDSH}
{\sc DeSalvo, S.}
\newblock Probabilistic divide-and-conquer: deterministic second half.
\newblock {\em arXiv preprint arXiv:1411.6698\/} (2014).

\bibitem{logconcave}
{\sc DeSalvo, S., and Pak, I.}
\newblock Log-concavity of the partition function.
\newblock {\em The Ramanujan Journal\/} (10 2013).

\bibitem{Boltzmann}
{\sc Duchon, P., Flajolet, P., Louchard, G., and Schaeffer, G.}
\newblock Boltzmann samplers for the random generation of combinatorial
  structures.
\newblock {\em Combinatorics, Probability and Computing 13}, 4-5 (2004),
  577--625.

\bibitem{ErdosLehner}
{\sc Erdos, P., and Lehner, J.}
\newblock The distribution of the number of summands in the partitions of a
  positive integer.
\newblock {\em Duke Mathematics Journal 8}, 2 (1941), 335--345.

\bibitem{esscher}
{\sc Esscher, F.}
\newblock On the probability function in the collective theory of risk.
\newblock {\em Skandinavisk Aktuarietidskrift 15\/} (1932), 175--195.

\bibitem{Feller}
{\sc Feller, W.}
\newblock {\em An {I}ntroduction to {P}robability {T}heory and {I}ts
  {A}pplications. {V}ol. {I}}.
\newblock John Wiley \& Sons, Inc., New York, N.Y., 1950.

\bibitem{Flajolet}
{\sc Flajolet, P., and Sedgewick, R.}
\newblock {\em Analytic combinatorics}.
\newblock Cambridge University Press, Cambridge, 2009.

\bibitem{Foata}
{\sc Foata, D.}
\newblock {\em La s\'erie g\'en\'eratrice exponentielle dans les probl\`emes
  d'\'enum\'eration}.
\newblock Les Presses de l'Universit\'e de Montr\'eal, Montreal, Que., 1974.
\newblock Avec un chapitre sur les identit{\'e}s probabilistes d{\'e}riv{\'e}es
  de la formule exponentielle, par B. Kittel, S{\'e}minaire de
  Math{\'e}matiques Sup{\'e}rieures, No. 54 ({\'E}t{\'e}, 1971).

\bibitem{FS}
{\sc Foata, D., and Sch{\"u}tzenberger, M.~P.}
\newblock {\em Th\'eorie g\'eom\'etrique des polyn\^omes eul\'eriens}.
\newblock Lecture Notes in Mathematics, Vol. 138. Springer-Verlag, Berlin-New
  York, 1970.

\bibitem{granville}
{\sc Granville, A.}
\newblock Cycle lengths in a permutation are typically {P}oisson.
\newblock {\em Electron. J. Combin. 13}, 1 (2006), Research Paper 107, 23.

\bibitem{HR18}
{\sc Hardy, G.~H., and Ramanujan, S.}
\newblock Asymptotic formulae in combinatory analysis.
\newblock {\em Proceedings of the London Mathematical Society\/} (1918), 75 --
  115.

\bibitem{saddlebook}
{\sc Jensen, J.~L.}
\newblock {\em Saddlepoint approximations}, vol.~16 of {\em Oxford Statistical
  Science Series}.
\newblock The Clarendon Press, Oxford University Press, New York, 1995.
\newblock Oxford Science Publications.

\bibitem{Jordan}
{\sc Jordan, C.}
\newblock {\em Calculus of finite differences}.
\newblock Third Edition. Introduction by Harry C. Carver. Chelsea Publishing
  Co., New York, 1965.

\bibitem{Joyal}
{\sc Joyal, A.}
\newblock Une th\'eorie combinatoire des s\'eries formelles.
\newblock {\em Advances in Mathematics 42}, 1 (1981), 1--82.

\bibitem{Joyal2}
{\sc Joyal, A.}
\newblock Foncteurs analytiques et esp\`eces de structures.
\newblock In {\em Combinatoire \'enum\'erative ({M}ontreal, {Q}ue.,
  1985/{Q}uebec, {Q}ue., 1985)}, vol.~1234 of {\em Lecture Notes in Math.}
  Springer, Berlin, 1986, pp.~126--159.

\bibitem{Lehmer38}
{\sc Lehmer, D.~H.}
\newblock On the series for the partition function.
\newblock {\em Transactions of the American Mathematical Society 43}, 2 (1938),
  271--295.

\bibitem{Lehmer39}
{\sc Lehmer, D.~H.}
\newblock On the remainders and convergence of the series for the partition
  function.
\newblock {\em Transactions of the American Mathematical Society 46\/} (1939),
  362--373.

\bibitem{Louchard1}
{\sc Louchard, G.}
\newblock Asymptotics of the {S}tirling numbers of the first kind revisited: a
  saddle point approach.
\newblock {\em Discrete Mathematics \& Theoretical Computer Science. DMTCS.
  12}, 2 (2010), 167--184.

\bibitem{Louchard2}
{\sc Louchard, G.}
\newblock Asymptotics of the {S}tirling numbers of the second kind revisited.
\newblock {\em Applicable Analysis and Discrete Mathematics 7}, 2 (2013),
  193--210.

\bibitem{MoserWyman2}
{\sc Moser, L., and Wyman, M.}
\newblock {S}tirling numbers of the second kind.
\newblock {\em Duke Mathematical Journal 25\/} (1957), 29--43.

\bibitem{MoserWyman1}
{\sc Moser, L., and Wyman, M.}
\newblock Asymptotic development of the {S}tirling numbers of the first kind.
\newblock {\em Journal of the London Mathematical Society. Second Series 33\/}
  (1958), 133--146.

\bibitem{Nicolas}
{\sc Nicolas, J.-L.}
\newblock Sur les entiers {$N$} pour lesquels il y a beaucoup de groupes
  ab{\'e}liens d'ordre {$N$}.
\newblock {\em Universit{\'e} de Grenoble. Annales de l'Institut Fourier 28}, 4
  (1978), 1--16, ix.

\bibitem{PittelShape}
{\sc Pittel, B.}
\newblock {O}n a likely shape of the random {F}errers diagram.
\newblock {\em Advances in Applied Mathematics 18}, 4 (1997), 432--488.

\bibitem{PittelSetPartitions}
{\sc Pittel, B.}
\newblock Random set partitions: asymptotics of subset counts.
\newblock {\em {J}ournal of Combinatorial Theory. Series A 79}, 2 (1997),
  326--359.

\bibitem{Rademacher}
{\sc Rademacher, H.}
\newblock A convergent series for the partition function $p(n)$.
\newblock {\em Proceedings of the National Academy of Sciences 23\/} (1937),
  78--84.

\bibitem{Reid}
{\sc Reid, N.}
\newblock Saddlepoint methods and statistical inference.
\newblock {\em Statistical Science. A Review Journal of the Institute of
  Mathematical Statistics 3}, 2 (1988), 213--238.
\newblock With comments and a rejoinder by the author.

\bibitem{RU}
{\sc Riddell{,}~Jr., R.~J., and Uhlenbeck, G.~E.}
\newblock On the theory of virial development of the equation of state of
  monoatomic gases.
\newblock {\em The Journal of Chemical Physics 21\/} (1953), 2056--2064.

\bibitem{Romik}
{\sc Romik, D.}
\newblock Partitions of $n$ into $t\sqrt{n}$ parts.
\newblock {\em European Journal of Combinatorics 26}, 1 (2005), 1--17.

\bibitem{rosser}
{\sc Rosser, J.~B., and Schoenfeld, L.}
\newblock Approximate formulas for some functions of prime numbers.
\newblock {\em Illinois J. Math. 6\/} (1962), 64--94.

\bibitem{Shepp}
{\sc Shepp, L.~A., and Lloyd, S.~P.}
\newblock Ordered cycle lengths in a random permutation.
\newblock {\em Transactions of the American Mathematical Society 121\/} (1966),
  340--357.

\bibitem{StanleyLogConcave}
{\sc Stanley, R.~P.}
\newblock Log-concave and unimodal sequences in algebra, combinatorics, and
  geometry.
\newblock In {\em Graph theory and its applications: {E}ast and {W}est
  ({J}inan, 1986)}, vol.~576 of {\em Ann. New York Acad. Sci.} New York Acad.
  Sci., New York, 1989, pp.~500--535.

\bibitem{Szekeres1}
{\sc Szekeres, G.}
\newblock An asymptotic formula in the theory of partitions.
\newblock {\em The Quarterly Journal of Mathematics. Oxford. Second Series 2\/}
  (1951), 85--108.

\bibitem{Szekeres2}
{\sc Szekeres, G.}
\newblock Some asymptotic formulae in the theory of partitions. {II}.
\newblock {\em The Quarterly Journal of Mathematics. Oxford. Second Series 4\/}
  (1953), 96--111.

\bibitem{vanLint}
{\sc van Lint, J.~H., and Wilson, R.~M.}
\newblock {\em A course in combinatorics}, second~ed.
\newblock Cambridge University Press, Cambridge, 2001.

\bibitem{Wilf}
{\sc Wilf, H.~S.}
\newblock {\em generatingfunctionology}.
\newblock Academic Press, Inc., Boston, MA, 1990.

\bibitem{Williams}
{\sc Williams, D.}
\newblock {\em Diffusions, {M}arkov processes, and martingales. {V}ol. 1},
  2~ed.
\newblock John Wiley \& Sons, Ltd., Chichester, 1979.
\newblock Foundations, Probability and Mathematical Statistics.

\end{thebibliography}

\end{document}